\documentclass{article}

\usepackage{graphicx}
\usepackage{amssymb,amsmath,amsthm}
\usepackage{multirow}
\usepackage{rotating}
\usepackage[english]{babel}

\newtheorem{condition}{Condition}
\newtheorem{theorem}{Theorem}
\newtheorem{lemma}{Lemma}

\usepackage{multirow}

\def\bh{\boldsymbol{h}}
\def\bZ{\boldsymbol{Z}}
\def\bz{\boldsymbol{z}}

\def\btheta{\boldsymbol{\theta}}

\def\bR{\boldsymbol{R}}
\def\blambda{\boldsymbol{\lambda}}

\def\ss{\boldsymbol{s}}

\def\trace{\mathrm{tr}}
\let\oldi\i
\renewcommand{\i}{\text{\oldi}}



\def\R{\mathbb{R}}
\def\S{\mathbb{S}}
\def\d{\textrm{d}}

\def\bh{\mathbf{h}}

\usepackage{color}
\definecolor{DarkGreen}{rgb}{0, 0.6, 0}
\input{standardmacros.sty}

\begin{document}

% "Title of the Paper"
\title{Asymptotically Equivalent Prediction in Multivariate Geostatistics}

\author{
{\bf Fran\c{c}ois Bachoc}  \\
Institut de math\'ematique, UMR5219; Universit\'e de Toulouse; \\
CNRS, UPS IMT, F-31062 Toulouse Cedex 9, France, \\
francois.bachoc@math.univ-toulouse.fr \\
{\bf Emilio Porcu} \\
School of Computer Science and Statistics,
Trinity College, Dublin, \\
porcue@tcd.ie \\
{\bf Moreno Bevilacqua} \\
Department of Statistics, University of Valparaiso, \\
Valparaiso, Chile, \\
 moreno.bevilacqua@uv.cl \\
{\bf Reinhard Furrer} \\
Department of Mathematics and \\ Department of Computational Science, \\ University of Zurich, 8057, Zurich, Switzerland, \\
reinhard.furrer@math.uzh.ch \\
{\bf Tarik Faouzi} \\
Department of Statistics, Universidad del Biobio, \\ 4081112, Conception, Chile, \\
tfaouzi@ubiobio.cl
}

\maketitle

\begin{abstract} \hspace{0.1cm}  
Cokriging is the common method of spatial interpolation (best linear unbiased prediction) in multivariate geostatistics. While best linear prediction has been well understood in univariate spatial statistics, the literature for the multivariate case has been elusive so far. The new challenges provided by modern spatial datasets, being typically multivariate, call for a deeper study of cokriging. In particular, we deal with the problem of misspecified cokriging prediction within the framework of fixed domain asymptotics.
 Specifically, we provide conditions  for  equivalence of measures associated with multivariate Gaussian random fields, with index set in a compact set of a $d$-dimensional Euclidean space. Such conditions have been elusive for over   about $50$ years of spatial statistics. \\
We then focus on the  multivariate Mat{\'e}rn and Generalized  Wendland  classes of matrix valued covariance functions, that have been  very popular for having parameters that are crucial to  spatial interpolation, and that control the mean square differentiability of the associated Gaussian process. We provide sufficient conditions, for equivalence of Gaussian measures, relying on the covariance parameters of these two classes. This enables to identify the parameters that are crucial to asymptotically equivalent interpolation in multivariate geostatistics. Our findings are then illustrated through simulation studies. \\
{\bf Keywords:} Cokriging, Equivalence of Gaussian Measures, Fixed Domain Asymptotics, Functional Analysis, Generalized Wendland, Mat{\'e}rn, Spectral Analysis
 \end{abstract}

\section{Introduction}

\subsection{Context}
%\textcolor{red}{We should rewrite this Section, in particular the final part}
Our paper deals with equivalence of Gaussian measures and asymptotically equivalent cokriging prediction  in multivariate geostatistics. We consider a multivariate ($p$-variate) stationary Gaussian field $\bZ=\{\bZ(\s) = ( Z_1(\s) , \ldots, Z_p(\s ))^\top, \ss \in D \},$ where $D $ is a fixed bounded subset of $\mathbb{R}^d$ with non-empty interior. Throughout, the integers $d$ and $p$ are fixed. The assumption of Gaussianity implies  that modeling, inference and prediction depend exclusively on the mean of $\bZ$, which is constant and assumed to be zero, and on the multivariate covariance function, being 
 a $p \times p$ matrix function  $\bR = [R_{ij}]_{i,j=1}^p$, defined  in $\R^d$, such  that
\[R_{ij}(\h) = \Cov( Z_i(\t) , Z_j(\s) ), \quad \h=\t - \s, \]
for $\t,\s \in D$ and $i,j=1,\ldots,p$. Throughout, the diagonal elements $R_{ii}$ are called marginal covariances, whereas the off-diagonal members $R_{ij}$ are called cross-covariances. The mapping $\bR$ must be positive definite, which means that 
\begin{equation} \label{posdef}
   \sum_{\ell=1}^{n} \sum_{k=1}^n   \boldsymbol{a}_{\ell}^\top \bR(\s_\ell - \s_k) \boldsymbol{a}_{k} \geq 0,
 \end{equation}
for all positive integer $n$, $\{\s_1,\hdots,\s_n\}\subset D$ and $\{\boldsymbol{a}_1,\hdots,\boldsymbol{a}_n\}\subset \mathbb{R}^p$. 

Spatial prediction in multivariate geostatistics is known as cokriging, which is the analogue  of best linear unbiased prediction in classical regression. Given $\boldsymbol{Z}$ as above and given observation locations $\{\s_1,\hdots,\s_n\}\subset D$, let $\bZ_{n}=(\bZ^{\top}_{1;n},\ldots,\bZ^{\top}_{p;n})^{\top}$ be the observation vector obtained from  $\bZ$,
where $\bZ_{i;n}=(Z_i(\ss_1),\ldots,Z_i(\ss_n) )^{\top}$, $i=1,\ldots,p$. Then the cokriging predictor of  the $j$-th component of  $\bZ$   at a target point $\s_0$, denoted $\widehat{Z}_{j;n}^{\text{CK}}(\s_0)$, is given by
\begin{equation*}
\widehat{Z}_{j;n}^{\text{CK}}(\s_0) = \mathbb{E} \left ( Z_j(\s_0) | \boldsymbol{Z}_n \right ), \quad  j=1,\ldots,p.
\end{equation*} 
The simple kriging predictor of $Z_j$ at a target point $\s_0$, denoted $\widehat{Z}_{j;n}^{\text{SK}}(\s_0)$, is instead given by $\widehat{Z}_{j;n}^{\text{SK}}(\s_0) = \mathbb{E} \left ( Z_j(\s_0) | \boldsymbol{Z}_{j;n} \right ).$
In this present paper, we put emphasis on the following problems: \\
{\bf A.} How important is the multivariate covariance function for spatial prediction? \\
{\bf B.} Which covariance parameters are important to cokriging? \\
We provide answers to these two questions and, in doing so, we obtain general sufficient conditions for equivalence of multivariate Gaussian measures, that are of independent interest. Notice that \cite{zhang2015doesn} have previously also addressed these two questions and provided more partial answers.
% \cite{zhang2015doesn} put emphasis on the following problems: \\
% {\bf A.} How important is the cross covariance function for spatial prediction? \\
% {\bf B.} Which parameters are important to cokriging? \\
% %{\bf C.} How much does cokriging outperform kriging? \\
% The authors provide a partial answer to  these questions. In doing so, they provide general sufficient conditions for equivalence of multivariate Gaussian measures, that are of independent interest. 

\subsection{Literature Review}

\subsubsection*{Multivariate Covariance Functions}
Multivariate covariance functions in $d$-dimensional Euclidean spaces have become ubiquitous and we refer the reader to \cite{Genton:Kleiber:2014} for a detailed account. Recently, there has been some work on multivariate covariance functions on non planar surfaces, and the reader is referred to \cite{APFM}, \cite{alegria2016dimple}, \cite{PBG16} and \cite{bevilacqua2019estimation}. %\textcolor{red}{Talking about sphere not necessary?}

As for constructive methods to provide new models, the linear model of coregionalization \cite{Wackernagel:2003} is based on representing any component of the multivariate field~$\boldsymbol{Z}$ as a linear combination of latent, uncorrelated fields.  Such a technique has been constructively criticized by \cite{Gneiting:Kleibler:Schlather:2010} and \cite{daley2015} as the  smoothness of any component of the multivariate  field amounts to that of the roughest underlying univariate process.  Moreover,  the number of parameters can quickly become massive as the number of components increases.  Scale mixture techniques as in \cite{Porcu20111293}, as well as latent dimension approaches \cite{Porcu:Gregori:Mateu:2006, Apanasovich:Genton:2010, Porcu20111293}  have been largely used to propose new multivariate models. \cite{bevilacqua2019estimation} call the following construction principle {\em multivariate parametric adaptation}: let  $\{ R(\cdot; \boldsymbol{\lambda}): [0, \infty) \to \mathbb{R}, \; \boldsymbol{\lambda} \in \mathbb{R}^k \}$ be a parametric family of continuous functions, such that $R(\|\cdot\|;\boldsymbol{\lambda})$ is a correlation function in $\R^d$ ($R(0; \boldsymbol{\lambda})=1$), indexed by a parameter vector $ \boldsymbol{\lambda}=(\lambda_1,\ldots,\lambda_k)^{\top}$. Call $ \boldsymbol{\lambda}_{ij}=(\lambda_{ij,1},\ldots,\lambda_{ij,k})^{\top}$, $i,j=1, \ldots, p$ a collection of parameter vectors in $\mathbb{R}^k$. Then, define $\boldsymbol{R}: [0,\infty) \to \mathbb{R}^{p \times p}$ through
\[
\boldsymbol{R}(x) = \left [ R_{ij}(x) \right ]_{i,j=1}^p , \qquad x \in [0,\infty), 
\]
with elements $R_{ij}$ defined as  
\begin{equation} \label{ed-sheeran}
 R_{ij}(x)= \sigma_{ii} \sigma_{jj} \rho_{ij} R(x; \boldsymbol{\lambda}_{ij} ), \qquad x \in [0,\infty), 
\end{equation} 
where $\sigma_{ii}^2$ is the variance of the $i$th component of the multivariate random field and where $\rho_{ii}=1$ and $\rho_{ij}$, $i\neq j$, is the colocated correlation coefficient. Thus, the problem is finding the restriction on the parameters $\boldsymbol{\lambda}_{ij}$ such that $\boldsymbol{R}(\|\cdot\|)$ is positive definite as in (\ref{posdef}).

A crucial benefit of this strategy, by comparison with  the linear model of coregionalization,  is a clear  physical interpretation of the parameters
\cite{bevilacqua2015}.
For example, for a bivariate random field, the  colocated correlation parameter, $\rho_{12}$, expresses the marginal correlation between the components, since
$R_{12}(0)=R_{21}(0)=\rho_{12}$ if $\sigma^2_{11}=\sigma^2_{22}=1$. In Euclidean spaces this strategy has been adopted by \cite{Gneiting:Kleibler:Schlather:2010}, \cite{doi:10.1080/01621459.2011.643197}  and by \cite{daley2015}.  

\subsubsection*{Misspecified Kriging Predictions under Infill Asymptotics}

The study of asymptotic properties of (co)kriging predictors
is complicated by the fact that more than one asymptotic framework
can be considered when observing a single realization from a (multivariate) Gaussian
field. Under infill asymptotics (also called fixed domain asymptotics), the typical assumption is that 
 the sampling domain is bounded and that the sampling set becomes
increasingly dense. Under increasing domain asymptotics, the sampling domain
increases with the number of observed data, and the distance between
any two observation locations is bounded away from zero \cite{bachoc2014asymptotic,Mardia:Marshall:1984}.

The focus of this paper is on infill asymptotics.
In this case, in the univariate case, a key concept is the equivalence of Gaussian measures \cite{Sko:ya:1973,Ibragimov-Rozanov:1978}. 
Furthermore, a long-standing object of attention is asymptotically optimal prediction when using a misspecified covariance function (the predictor is then called pseudo BLUP by Michael Stein \cite{Stein:1999}). In the univariate case, Michael Stein has shown that, when the Gaussian measures obtained from the true and misspecified covariance function are equivalent, then the predictions under the misspecified
covariance function are asymptotically efficient, and mean square errors are asymptotically equivalent to their targets \cite{Stein:1988,Stein:1990,Stein:1993,Stein:1999b,Stein:2004}.

When working with specific covariance models, it is thus crucial to know which conditions on the parameters imply the equivalence of Gaussian measures. Specific results have been provided for the Mat{\'e}rn \cite{Zhang:2004} and Generalized Wendland \cite{bevilacqua2019estimation} classes of covariance functions, associated with scalar valued random fields. These results themselves follow from earlier works on general conditions for equivalence of univariate Gaussian measures, in particular based on spectral densities \cite{Sko:ya:1973}.
Nevertheless, multivariate extensions of these various results are lacking. They are provided in the present paper.

\subsection{Outline}
While best linear prediction has been well understood in univariate spatial statistics, as we have discussed above, the literature for the multivariate case has been elusive so far. Nevertheless, the new challenges provided by modern spatial datasets, being typically multivariate, call for a deeper study of cokriging. This is the object of this paper, where we deal with the problem of misspecified cokriging prediction within the framework of infill asymptotics.  

It turns out that the contributions related to equivalence of measures for Gaussian random fields are limited, since the early $70$ies, to scalar valued random fields, see the above discussion. Hence, to study cokriging prediction under fixed domain asymptotics, it is imperative to understand equivalence of measures for multivariate random fields defined over compact sets in $\R^d$. This paper provides a solution to the problem, by providing general sufficient conditions for equivalence. 

We also focus on the  multivariate Mat{\'e}rn and Generalized  Wendland  classes of matrix valued covariance functions, that have been  very popular in spatial statistics for having parameters that are crucial to  spatial interpolation, and that control the mean square differentiability of the associated Gaussian process. We show parametric conditions ensuring these matrix valued covariance models to be compatible, that is, to yield equivalent Gaussian measures. Hence, we provide sufficient conditions for asymptotic equivalence of misspecified cokriging predictions. We confirm and illustrate this asymptotic equivalence numerically.

The outline of the paper is the following: Section~\ref{sec2} contains the necessary mathematical and probabilistic background.
Section \ref{SECTION:GENERAL:CONDITIONS:EQUIVALENCE} contains general results about compatible matrix valued covariance functions.
Section \ref{SECTION:APPLICATION:MATERN:WENDLAND} relates on the compatibility between the Mat\'ern and Generalized Wendland parametric classes of bivariate covariance functions. Section~\ref{5} inspects the problem of cokriging predictions through these models. 
Our findings are then illustrated through a simulation study in Section~\ref{sec:numill}.
The proofs are lengthy and technical, so that we deferred those to the Appendix to favor a neater exposition.

\section{Background and Notation} \label{sec2} 

\subsection{Multivariate Covariance Functions  and Function Spaces}

Let  $d,p$ be positive integers. Let $\bR: \R^d \to  \R^{p \times p}$ be positive definite. We let the elements $R_{ij}$ of $\bR$ be continuous in $\R^d$. 
 The matrix spectral density of $\bR$ is the 
$p\times p$ matrix function $\bF = [F_{ij}]_{i,j=1}^p$ defined by 
\[
R_{ij}(\h) = \int_{\IR^d} F_{ij}(\blambda) e^{ \i \h^\top \blambda } d \blambda,
\]
for $\h \in \mathbb{R}^d$ and $i,j=1,\ldots,p$. Here $\i$ is the complex number satisfying $\i^2 = -1$. Note  that a sufficient condition for $\bF$ to be well defined is that $\bR$ has elements $R_{ij}$ that are pointwise absolutely integrable in $\R^d$, and that the same holds for the Fourier transforms of these elements.

For $a=0,1$, we consider a stationary matrix covariance function $\bR^{(a)} = [R^{(a)}_{ij}]_{i,j=1}^{p}$ on $\mathbb{R}^d$. We assume that, for $a = 0,1$ and $i,j=1,\ldots,p$, the function $R^{(a)}_{ij}$ is summable on $\mathbb{R}^d$ and that $\bR^{(a)}$ has matrix spectral density $\bF^{(a)} = [F^{(a)}_{ij}]_{i,j=1}^{p}$.

We further assume that for $a=0,1$ and $j=1,\ldots,p$, $F^{(a)}_{jj}$ is real-valued, strictly positive on $\IR^d$ and summable on $\IR^d$. We remark that for $a=0,1$ and $i,j=1,\ldots,p$, $i \neq j$, $F^{(a)}_{ij}$ is complex-valued and we also assume that $|F^{(a)}_{ij}|$ is summable on $\IR^d$, with $|z|$ the modulus of $z \in \IC$. 
Cram\'er's theorem shows that, for any $\blambda \in \IR^d$ and $a=0,1$, the matrix $\bF^{(a)}(\blambda)$ is Hermitian with non-negative eigenvalues.

For a $p \times p$ Hermitian matrix $\bfM$, we let $\lambda_1(\bfM) \leq  \dots \leq \lambda_p(\bfM)$ be its $p$ eigenvalues. 
If $\bfM$ is non-negative definite, we let $\bfM^{1/2}$ be its unique Hermitian non-negative definite square root.
For a square complex matrix $\N$, we let $\|\N\|$ be its largest singular value.
For a complex column vector $\bv$, we let $\bar{\bv}$ be composed of the conjugates of $\bv$ and $\|\bv\|^2 = \bar{\bv}^\top \bv$.
For two $p \times p$ Hermitian matrices $\bfM$ and $\N$ we write $\bfM \geq \N$ when for all $\bv \in \IC^p$, $\bar{\bv}^\top \bfM \bv \geq \bar{\bv}^\top \N \bv$.
We let $\bfe_1,\ldots,\bfe_q$ be the $q$ basis column vectors of $\IR^q$ for $q \in \IN$.

For a summable function $f : \IR^d \to \IR$, we let the Fourier transform $\cF(f)$ of $f$ be defined by, for $\blambda \in \mathbb{R}^d$, 
\[
\cF(f)(\blambda) = \frac{1}{(2 \pi)^d} \int_{\IR^d}  f( \t) e^{ - \i \blambda^\top \t} d \t.
\]

For a bounded subset $S$ of $\mathbb{R}^d$ with non-empty interior, we let $\cW_{S}$ be the set of functions from $\IR^d$ to $\IC^p$ of the form $(f_1,\ldots,f_p)^\top$, where for $i=1,\ldots,p$, $f_i=\cF(g_i)$ for a function $g_i$ in $L^2(\IR^d)$ that is zero outside of $S$. As observed in \cite{Sko:ya:1973}, a function $(f_1,\ldots,f_p)^\top$ in $\cW_{S}$ satisfies $\int_{\IR^d} (|f_1(\blambda)|^2 + \dots + |f_p(\blambda)|^2) d \blambda < \infty$.
Consider a matrix function $\blambda \mapsto \bF(\blambda)$ with $\blambda \in \IR^d$ and with $\bF(\blambda)$ a $p \times p$ Hermitian strictly positive definite matrix and assume that $\|\bF\|$ and $\lambda_p(\bF) / \lambda_1(\bF)$ are bounded on $\IR^d$.
Then we define $\cW_{S}( \bF )$ as the closure of $\cW_{S}$ in the metric
\[
\| \bof \|^2_{\cW_{S}(\bF)} = \int_{\IR^d} \bar{\bof}(\blambda)^\top \bF(\blambda) \bof (\blambda) d \blambda.
\]
We remark that $\cW_{S}( \bF )$ is a (complex) separable Hilbert space, with inner-product given by 
\[
( \bof_1 , \bof_2 )_{\cW_{S}( \bF )} = \int_{\IR^d} \bar{\bof}_1(\blambda)^\top \bF(\blambda) \bof_2 (\blambda) d \blambda.
\]
Indeed, for $\bof = ( f_1,\ldots,f_p )^\top \in \cW_{S}( \bF )$, for $i=1,\ldots,p$, $f_i$ is included in the space of square integrable functions w.r.t. the measure $\|\F(\blambda)\| d \blambda$ which is separable and complete.

For $S \subset \IR^d$, we let $L^{2,p}_{S}$ be the Hilbert space of the vectors of functions of the form $(f_1,\ldots,f_p)^\top$, with $f_i : S \to \IC$ square summable for $i=1,\ldots,p$, endowed with the inner product
\[
( (f_1,\ldots,f_p)^\top , (g_1,\ldots,g_p)^\top ))_{L^{2,p}_{S}}
=
\int_S \bar{f}_1(\t) g_1(\t) d \t
+
\dots
+
\int_S \bar{f}_p(\t) g_p(\t) d \t.
\]

\subsection{The Univariate Mat{\'e}rn and Generalized Wendland Covariance Functions}

We start by describing the two univariate classes of covariance functions that will be used throughout as building blocks for matrix valued covariance functions. \\
{\bf 1.} The Mat{\'e}rn function \cite{Stein:1999} is defined as:
\begin{equation}
%\label{Mat{\'e}rn }
{\cal M}_{\nu,\alpha}(r)=
\frac{2^{1-\nu}}{\Gamma(\nu)} \left (\frac{r}{\alpha}
  \right )^{\nu} {\cal K}_{\nu} \left (\frac{r}{\alpha} \right ),
  \qquad r \ge 0, \label{matern-cov}
\end{equation}
where $r = \|\boldsymbol{h}\|$, $\boldsymbol{h} \in \mathbb{R}^d$. The Mat{\'e}rn covariance function is positive definite in $\R^d$ for all positive $\alpha$ and for any value of $d$.  Here,
${\cal K}_{\nu}$ is a modified Bessel function of the second kind of
order $\nu$.
The parameter $\nu>0$ characterizes the
differentiability at the origin and, as a consequence,  the
differentiability of the associated sample paths. In particular for a
positive integer $k$, the sample paths are $k$ times differentiable, in any direction, if
and only if $\nu>k$. When $\nu=1/2+m$ and $m$ is a nonnegative integer, the Mat{\'e}rn
function simplifies to the product of a negative exponential with a
polynomial of degree $m$, and for $\nu$ tending to infinity, a rescaled
version of the Mat{\'e}rn converges to a squared exponential model being
infinitely differentiable at the origin. Thus, the Mat{\'e}rn function
allows for a continuous parameterization of its associated Gaussian
field in terms of smoothness.  \\
{\bf 2.} The Generalized Wendland function \cite{Gneiting:2002,zastavnyi2006some} is defined, for $\kappa, \beta>0$, as
\begin{equation} \label{WG2*}
{\cal W}_{\mu,\kappa,\beta}(r):= \begin{cases}  \frac{1}{B(2\kappa,\mu+1)} \int_{r/{\beta}}^{1} u(u^2-(r/\beta)^2)^{\kappa-1} (1-u)^{\mu}\,\d u  ,& 0 \leq r/\beta < 1,\\ 0,& r/\beta \geq 1, \end{cases}
\end{equation}
with $B$ denoting the beta function, and where $r = \|\h\|$, $\h \in \IR^d$. The function  ${\cal W}_{\mu,\kappa,\beta}(r)$ is positive definite in $\mathbb{R}^d$ 
if and only if
\begin{equation} \label{lawea}
\mu \ge (d+1)/2+ \kappa.
\end{equation}

Note that  ${\cal W}_{\mu,0,\beta}$ is not defined because $\kappa$ must be strictly positive. In this special case 
we consider
the Askey function \cite{Askey:1973} 
%${\cal A}_{\mu}: [0,\infty) \to \R$, defined by
\begin{equation*} 
{\cal A}_{\mu,\beta}(r) :=   \begin{cases}   \left ( 1- r/\beta \right )^{\mu} ,& 0 \leq r/\beta< 1,\\ 0,&r/\beta \geq 1. \end{cases}
\end{equation*}
Arguments in \cite{Golubov:1981} show that ${\cal A}_{\mu,\beta}$ is positive definite if and only if $\mu \ge (d+1)/2$
and   we define ${\cal W}_{\mu,0,\beta}:= {\cal A}_{\mu,\beta}$. 

Closed
form solution of the integral in \eqref{WG2*} can be obtained when
$\kappa=k$, a positive integer.  In this  case,
 $ {\cal W}_{\mu,k,\beta}(r) =
{\cal A}_{\mu+k,\beta}(r) P_{k}(r)$,
with $P_{k}$ a polynomial of order $k$.
These functions, termed (original) Wendland functions, were originally proposed by   \cite{Wendland:1995}.

Other closed form solutions of integral~\eqref{WG2*} can be obtained when $\kappa=k+0.5$, using some results in \cite{Schaback:2011}. Such solutions are called \emph{missing Wendland}  functions.

As noted by  \cite{Gneiting:2002b}, Generalized Wendland  and Mat{\'e}rn functions exhibit the same behavior at the origin, with the smoothness parameters of the two covariance models related by the equation $\nu=\kappa+1/2$.

Here, for a positive integer $k$, the sample paths  of a Gaussian field with Generalized Wendland  function are $k$ times differentiable, in any direction, if and only if  $\kappa >k-1/2$.

\subsection{Bivariate Mat{\'e}rn and Generalized Wendland Models} 

We now consider the multivariate parametric adaptation, illustrated through Equation~\eqref{ed-sheeran}, as a construction principle for multivariate covariance functions. For simplicity, we focus on the case $p=2$ (bivariate Gaussian random fields) but the results following subsequently can be namely extended to $p>2$. 

We thus follow \cite{Gneiting:Kleibler:Schlather:2010}  and \cite{daley2015} to couple construction (\ref{ed-sheeran}) with, respectively, the Mat{\'e}rn model~\eqref{matern-cov} and the Generalized Wendland model (\ref{WG2*}), to obtain: \\
{\bf 1. The Bivariate Mat{\'e}rn model}, denoted ${\cal BM}_{\btheta}$, and defined as 
\begin{equation}\label{eq:piip5}
{\cal BM}_{\btheta}(r)= \left[ \rho_{ij}\sigma_{ii}\sigma_{jj} {\cal M}_{\nu,\alpha_{ij}}(r) \right]_{i,j=1}^2 , \quad  \rho_{11}=\rho_{22} = 1, \qquad \alpha_{12} =  \alpha_{21},
\end{equation}
where $\btheta=(\sigma_{11},\sigma_{22},\rho_{12},\nu,\alpha_{11},\alpha_{22},\alpha_{12})^{\top}$; \\
{\bf 2. The Bivariate Generalized Wendland model}, denoted ${\cal BW}_{\blambda}$, and defined as   
\begin{equation}\label{eq:piip6}
{\cal BW}_{\blambda}(r)= \left[ \rho_{ij}\sigma_{ii}\sigma_{jj} {\cal W}_{\mu,\kappa,\beta_{ij}}(r) \right]_{i,j=1}^2  , \quad  \rho_{11}=\rho_{22} = 1, \qquad \beta_{12} =  \beta_{21},
\end{equation}
where
$\blambda=(\sigma_{11},\sigma_{22},\rho_{12},\mu,\kappa,\beta_{11},\beta_{22},\beta_{12})^{\top}$.

Note that, in principle, the smoothness parameters  $\nu$ and  $\kappa$ for both models can change through the components. Nevertheless, in this paper we assume common smoothness parameters.

Henceforth, for the bivariate Mat{\'e}rn, we assume the following condition on the colocated correlation parameter:
\begin{equation}
\rho_{12}^2
<
\frac{  
  \alpha_{12}^{4 \nu}
}{
\alpha_{11}^{2 \nu}  \alpha_{22}^{2 \nu}
}
\inf_{t \geq 0}
\frac{
(  \alpha_{12}^{-2} +t^2 )^{2 \nu +d}
}{
(  \alpha_{11}^{-2} +t^2 )^{\nu +d/2}
(  \alpha_{22}^{-2} +t^2 )^{\nu +d/2}
}.
\end{equation}
This condition guarantees that the bivariate Mat{\'e}rn is valid, that is positive definite  \cite{Gneiting:Kleibler:Schlather:2010}.
Similarly, for the bivariate Generalized Wendland model we assume that

\begin{equation} \label{eq:cond:inf:wendland:longer}
\rho_{12}^2
<\frac{\beta_{11}^{d}\beta_{22}^{d}}{\beta_{12}^{2d}}
\inf_{z \geq 0}
\frac{
\mathstrut_1 F_2\Big(\zeta;\zeta+\frac{\mu}{2},\zeta+\frac{\mu}{2}+\frac{1}{2};-\frac{(z\beta_{11})^{2}}4\Big) \mathstrut_1 F_2\Big(\zeta;\zeta+\frac{\mu}{2},\zeta+\frac{\mu}{2}+\frac{1}{2};-\frac{(z\beta_{22})^{2}}4\Big)
}{
\left(\mathstrut_1 F_2\Big(\zeta;\zeta+\frac{\mu}{2},\zeta+\frac{\mu}{2}+\frac{1}{2};-\frac{(z\beta_{12})^{2}}4\Big)
\right)^2},
\end{equation}
where $\zeta=(d+1)/2+\kappa$ and
\begin{equation} \label{eq:oneFtwo}
\mathstrut_1 F_2(a;b,c;z)=\sum_{k=0}^{\infty}\frac{(a)_{k}z^{k}}{(b)_{k}(c)_{k}k!}, \qquad z \in\mathbb{R},
\end{equation}
is a special case of the generalized hypergeometric functions $\mathstrut_q F_p$ \cite{Abra:Steg:70}, with $(q)_{k}=  \Gamma(q+k)/\Gamma(q)$ for $k \in  \mathbb{N}\cup\{ 0\}$, being the Pochhammer symbol.
These technical conditions will be carefully explained in the Appendix. 

\subsection{Equivalence of Gaussian Measures and Cokriging}
Equivalence and orthogonality of probability measures are useful tools when assessing the asymptotic properties of both prediction and estimation for Gaussian fields. 
We denote with $P^{(a)}$, $a=0,1$, two probability measures defined on the same measurable space $\{\Omega, \cal F\}$. The measures $P^{(0)}$ and $P^{(1)}$ are called equivalent (denoted $P^{(0)} \equiv P^{(1)}$) if, for any $A\in \cal F$, $P^{(1)}(A)=1$  implies $P^{(0)}(A)=1$ and vice versa. On the other hand,  $P^{(0)}$ and $P^{(1)}$ are orthogonal if there exists an event $A$ such that $P^{(1)}(A)=1$ but $P^{(0)}(A)=0$. For a $p$-variate Gaussian random field $ \bZ : \Omega \times D \to \mathbb{R}^p$, to define previous concepts, we restrict the event $A$ to the $\sigma$-algebra generated by $\bZ$ and we emphasize this restriction by saying that the
two measures are equivalent on the paths of $\bZ$.  It is well known that two Gaussian measures (that is two measures on $\Omega$ such that $\bZ$ is Gaussian) are either equivalent or orthogonal on the paths of $\bZ$ \cite{Ibragimov-Rozanov:1978}.

Since a Gaussian measure is completely characterized by the mean function and matrix covariance function, we write $P(\bR)$ for a Gaussian measure  on  $(\Omega, \cal  F)$ such that $\bZ$ has zero mean and matrix covariance function $\bR$.
We also write  $P(\bR^{(0)}) \equiv P(\bR^{(1)}) $ on the paths of $\bZ$, if   two Gaussian measures with mean zero and the matrix covariance functions $\bR^{(0)}$ and $\bR^{(1)}$ are equivalent on the paths of $\bZ$. \\
For the remainder of the paper, we call $\bR^{(0)}$ and $\bR^{(1)}$ {\em compatible} when $P(\bR^{(0)}) \equiv P(\bR^{(1)}) $ on the paths of $\boldsymbol{Z}$. \\
A direct implication of the celebrated result by \cite{blackwell1962merging} is that, if two matrix valued covariance functions are compatible, then the two cokriging predictors are asymptotically equivalent (under fixed domain asymptotics). 

\section{General Results} \label{SECTION:GENERAL:CONDITIONS:EQUIVALENCE}

Let us consider two matrix covariance  functions $\bR^{(0)}$ and $\bR^{(1)}$ with associated matrix  spectral densities $\bF^{(0)}$ and $\bF^{(1)}$
and let $P(\bR^{(0)})$ and $P(\bR^{(1)})$ be the associated Gaussian measures. The next condition is our general technical requirement on $\bR^{(0)}$ and $\bR^{(1)}$. As shown in Appendix~\ref{app:cond},
%\ref{SECTION:APPLICATION:MATERN:WENDLAND}, 
this condition holds for the Mat\'ern and Generalized Wendland covariance functions.

\begin{condition} \label{cond:phi:zero}
There exist two constants $0 < c_1 < c_2 < \infty$ and a function $\gamma: \IR^d \to \IR$ such that $(\gamma,\ldots,\gamma)^\top$ is in $\cW_{[-b,b]^d}$ for some fixed $0 < b < \infty$ and
\[
c_1 \gamma^2 \I_p \leq \bF^{(0)} \leq c_2 \gamma^2 \I_p,
\]
\[
c_1 \gamma^2 \I_p \leq \bF^{(1)} \leq c_2 \gamma^2 \I_p.
\]
\end{condition}

We now provide a fundamental result for this paper. It relates about a sufficient condition for the compatibility of $\bR^{(0)}$ and $\bR^{(1)}$. It is an extension of Theorem~1 in \cite{Sko:ya:1973} from the univariate to the multivariate case.

\begin{theorem} \label{theorem:sufficient:representation:b}
Assume that Condition \ref{cond:phi:zero} holds, and 
 that there exists a matrix-valued function $\bB$ on $(\IR^d)^2$, such that for $\blambda , \bmu \in \IR^d $, $\bB(\blambda,\bmu) = [ b(\blambda,\bmu)_{ij} ]_{i,j=1}^p$ is a $p \times p$ complex matrix. 
Let $\bB(\blambda,\bmu) = \bar{\bB}(\bmu,\blambda)^\top$ for all $\blambda , \bmu \in \IR^d$.
Assume also that for $i,j=1,\ldots,p$, we have
\begin{equation} \label{eq:the:finite:integral}
\int_{\IR^d}
\int_{\IR^d}
|b_{ij}(\blambda , \bmu)|^2\,
\| \bF^{(0)}(\blambda) \|
\,
\| \bF^{(0)}(\bmu) \|
d \blambda 
d  \bmu 
< + \infty.
\end{equation}
Assume then that we have, for $\t,\s \in D$ and $\h=\t-\s$,
\begin{equation} \label{eq:assumption:Rdtwo:minus:Rone}
\bR^{(1)}(\h)
-
\bR^{(0)}(\h)
=
\int_{\IR^d}
\int_{\IR^d}
e^{ - \i \blambda^\top \t + \i \bmu^\top \s }
\bF^{(0)}(\blambda)
\bB(\blambda , \bmu) 
 \bF^{(0)}(\bmu)
d \blambda 
d \bmu.
\end{equation}
Then,  $P(\bR^{(0)})\equiv P(\bR^{(1)})$ on the paths of $\bZ$. 
\end{theorem}

Notice that the function in the integral \eqref{eq:assumption:Rdtwo:minus:Rone} is summable because of Cauchy-Schwarz inequality in concert with \eqref{eq:the:finite:integral} and Condition \ref{cond:phi:zero}.

We now provide a second fundamental result  for this paper. The proof of Theorem~\ref{theorem:sufficient:integral:normalized:difference:small} relies on Theorem~\ref{theorem:sufficient:representation:b}.
Theorem~\ref{theorem:sufficient:integral:normalized:difference:small} is particularly well applicable to specific models of covariance functions, as it just requires to show that two matrix spectral densities are sufficiently close for large frequencies. This enables us to address the Mat\'ern and Generalized Wendland models in Section~\ref{SECTION:APPLICATION:MATERN:WENDLAND}. 
Theorem~\ref{theorem:sufficient:integral:normalized:difference:small} is an extension of Theorem~4 in \cite{Sko:ya:1973} from the univariate to the multivariate case.

\begin{theorem} \label{theorem:sufficient:integral:normalized:difference:small}
Assume that Condition \ref{cond:phi:zero} holds, and that
\[
\int_{ \IR^d }
\frac{1}{\gamma(\blambda)^4}
\bnorml
\bF^{(0)}(\blambda ) - \bF^{(1)}(\blambda )
\bnormr^2
d \blambda
< \infty.
\]
Then,  $P(\bR^{(0)})\equiv P(\bR^{(1)})$ on the paths of $\bZ$.
\end{theorem}

We remark that the extensions of Theorems 1 and 4 in \cite{Sko:ya:1973}, from the univariate to the multivariate case, require substantial original proof techniques. As such, the proofs of Theorems \ref{theorem:sufficient:representation:b} and \ref{theorem:sufficient:integral:normalized:difference:small} are postponed to the Appendix.

\section{Compatible  Bivariate Mat{\'e}rn and Generalized Wendland Correlation Models}\label{SECTION:APPLICATION:MATERN:WENDLAND}

Let us consider the  parameter vectors $\btheta^{(a)}=(\sigma^{(a)}_{11},\sigma^{(a)}_{22},\rho^{(a)}_{12},\nu,\alpha^{(a)}_{11},\alpha^{(a)}_{22},\alpha^{(a)}_{12})^{\top}$
and $\blambda^{(a)}=(\sigma^{(a)}_{11},\sigma^{(a)}_{22},\rho^{(a)}_{12},\mu,\kappa,\beta^{(a)}_{11},\beta^{(a)}_{22},\beta^{(a)}_{12})^{\top}$, for $a=0,1$.

Our first result gives sufficient conditions for the  compatibility of two bivariate Mat{\'e}rn models with a common  smoothness parameter.

\begin{theorem} \label{theorem:matern:vs:matern}
Let $\nu>0$. If
\begin{equation}\label{condmat}
 \frac{\sigma^{(0)}_{ii}\sigma^{(0)}_{jj}\rho^{(0)}_{ij}}{(\alpha^{(0)}_{ij})^{2\nu}}  = \frac{\sigma^{(1)}_{ii}\sigma^{(1)}_{jj}\rho^{(1)}_{ij}}{(\alpha^{(1)}_{ij})^{2\nu}},  \quad i,j=1,2,
\end{equation}
then for $d=1, 2, 3$, the bivariate matrix valued covariance models 
${\cal BM}_{\btheta^{(0)}}$ and  ${\cal BM}_{\btheta^{(1)}}$ are compatible. 
\end{theorem}

Some comments are in order.
For each of the four pairs of covariance or cross covariance functions, the equality condition \eqref{condmat} is the same as in the univariate case \cite{Zhang:2004}.
\cite{zhang2015doesn} provide conditions for compatibility of ${\cal BM}_{\btheta^{(0)}}$ and  ${\cal BM}_{\btheta^{(1)}}$ for a very special case, where $\alpha_{ij}^{(a)}=\alpha^{(a)}>0$ for all $i,j=1,2$ and $a=0,1$. Hence, the authors consider two separable models.  Thus, Theorem \ref{theorem:matern:vs:matern} allows for a considerable improvement with respect to  \cite{zhang2015doesn} as it allows for different range parameters among the covariance and cross covariance functions. In the case where $\alpha^{(a)}_{11} = \alpha^{(a)}_{22} = \alpha^{(a)}_{12}$ for $a=0,1$, Theorem \ref{theorem:matern:vs:matern} coincides with \cite{zhang2015doesn}.

Our second result gives sufficient conditions for the  compatibility of two bivariate Generalized Wendland models with a common  smoothness parameter.

\begin{theorem} \label{theorem:wendland:vs:wendland}
For a given $\kappa \ge 0$,  let $\mu > d+1/2+\kappa$. If 
\begin{equation} \label{condition1_iff}
\frac{\sigma^{(0)}_{ii}\sigma^{(0)}_{jj}\rho^{(0)}_{ij}}{(\beta^{(0)}_{ij})^{1+2\kappa}} = \frac{\sigma^{(1)}_{ii}\sigma^{(1)}_{jj}\rho^{(1)}_{ij}}{(\beta^{(1)}_{ij})^{1+2\kappa}},\quad i,j=1,2,
\end{equation}
then for $d=1, 2, 3$, the bivariate matrix valued covariance models 
${\cal BW}_{\blambda^{(0)}}$ and ${\cal BW}_{\blambda^{(1)}}$ are compatible.
\end{theorem}

Our third result gives sufficient conditions for the  compatibility of a bivariate Mat{\'e}rn model with a  bivariate Generalized Wendland.
To simplify notation, we let $\alpha_{ij}^{(0)} = \alpha_{ij}$ and $\beta_{ij}^{(1)} = \beta_{ij}$ for $i,j=1,2$.

\begin{theorem}\label{theorem:matern:vs:wendland}
For given $\nu\geq1/2$ and $\kappa \geq 0$,
   if $\nu=\kappa+1/2$, $\mu > d+1/2+\kappa$,  and
\begin{equation}\label{cafu}
 \frac{\sigma^{(0)}_{ii}\sigma^{(0)}_{jj}\rho^{(0)}_{ij}}{\alpha^{2\nu}_{ij}}
 =
C_{\kappa,\mu}
  \frac{\sigma^{(1)}_{ii}\sigma^{(1)}_{jj}\rho^{(1)}_{ij}}{\beta_{ij}^{1+2\kappa}},\quad i,j=1,2,
\end{equation}
$
C_{\kappa,\mu}={\mu\Gamma(2\kappa+\mu+1)}\big/{\Gamma(\mu+1)}$
then for $d=1, 2, 3$, the bivariate matrix valued covariance models ${\cal BM}_{\btheta^{(0)}}$ and
 ${\cal BW}_{\blambda^{(1)}})$ are compatible.
\end{theorem}

Theorems \ref{theorem:wendland:vs:wendland} and \ref{theorem:matern:vs:wendland} have no existing counterpart, even in the restricted setting where $\alpha_{ij}^{(a)}=\alpha^{(a)}$
and $\beta_{ij}^{(a)}=\beta^{(a)}$
for all $i,j=1,2$ and $a=0,1$. Again, for each pair of covariance or cross covariance functions, the conditions \eqref{condition1_iff} and \eqref{cafu} on the covariance parameters are the same as in the univariate case in \cite{bevilacqua2019estimation} (see the proof of Theorem \ref{theorem:matern:vs:wendland} that relates $C_{\kappa,\mu}$ to the constants used in \cite{bevilacqua2019estimation}).

\section{Cokriging Predictions with Bivariate Generalized Wendland and Mat{\'e}rn   models}\label{5}

We now consider  prediction at a new target location $\ss_0$ given  a realization of a zero mean bivariate Gaussian field, using the bivariate Mat{\'e}rn and Generalized Wendland model, under fixed domain asymptotics.
Specifically, we focus on two properties: asymptotic efficiency of  prediction and asymptotically correct estimation of prediction variance.
\cite{Stein:1988}, in the univariate case, shows that both asymptotic
properties hold when kriging prediction is performed with two  equivalent Gaussian measures.

Let  $\{ \ss_{i,1},\ldots,\ss_{i,n_i} \in D \subset \mathbb{R}^d \}$, $i=1,2$, be any two sets of two-by-two distinct observation locations.
Let $\bZ_{n_1,n_2}=(\bZ^{\top}_{1;n_1},\bZ^{\top}_{2;n_2})^{\top}$ be the observation vector obtained from a bivariate Gaussian field
$\{ \bZ(\ss)=(Z_1(\ss),Z_2(\ss))^{\top},\ss \in D  \}$, 
where $\bZ_{i;n_i}=(Z_i(\ss_{i,1}),\dots,\linebreak[1] Z_i(\ss_{i,n_i}) )^{\top}$, $i=1,2$. Remark that we do not necessarily assume collocated observation locations, that is the sets $\{ \ss_{1,1},\ldots,\ss_{1,n_1} \}$ and $\{ \ss_{2,1},\ldots,\ss_{2,n_2} \}$ can be different.

Suppose we want to predict the first of the two components of the  bivariate random field 
at $\ss_0$ that is  $Z_1(\ss_0)$, $\ss_0\in D$, using  a misspecified bivariate model. For simplicity we only consider prediction for the first as symmetrical arguments hold for the second component.
%assuming that the true covariance  model is  different. 
Specifically, we denote with $\mathcal{BC}^{(a)}=[C^{(a)}_{ij}]_{i,j=1}^{2}$, $a=0,1$ the true and misspecified  bivariate matrix covariance function respectively.

Let $\boldsymbol{c}^{(a)}_{1;n_1,n_2}=(\boldsymbol{c}^{(a)}_{1;1;n_1}{\!}^{\top},\boldsymbol{c}^{(a)}_{2;1;n_2}{\!}^{\top})^{\top}$ with 
$\boldsymbol{c}^{(a)}_{i;1;n_i}=[C^{(a)}_{1i}(\|\boldsymbol{s}_0-\boldsymbol{s}_{i,\ell} \|)]_{\ell=1}^{n_i}$, $i=1,2$,
the vector covariances between the location to predict and  $\bZ_{n_1,n_2}$. Let also
$\C^{(a)}_{n_1,n_2}$ be the $(n_1+n_2) \times (n_1+n_2)$ matrix, with block $i,j$, of size $n_i \times n_j$, given by $[C^{(a)}_{ij}(\|\boldsymbol{s}_{i,\ell}-\boldsymbol{s}_{j,k} %\|)]_{l=1,\ldots,n_i,m=1,\ldots,n_j}$, $i,j=1,2$,
\|)]_{\ell=1,k=1}^{n_i,n_j}$, $i,j=1,2$,
the variance-covariance matrix associated with $\bZ_{n_1,n_2}$.

The (misspecifed) optimal predictor of $Z_1(\ss_0)$, using $\mathcal{BC}^{(0)}$ and  $\mathcal{BC}^{(1)}$, is given by, 
$$
\widehat{Z}^{\mathcal{BC}^{(a)}}_{1;n_1,n_2}(\ss_0)=\boldsymbol{c}^{(a)}_{1;n_1,n_2}{\!}^{\top}(\C^{(a)}_{n_1,n_2})^{-1} \bZ_{n_1,n_2}.$$
Under the  correct model $\mathcal{BC}^{(0)}$,  the mean squared prediction error based on $\mathcal{BC}^{(1)}$ is given by
\begin{align}\label{mse_miss3}\begin{split}
\var_{\mathcal{BC}^{(0)}}\bigl[\widehat{Z}^{\mathcal{BC}^{(1)}}_{1;n_1,n_2}(\ss_0)-Z_1(\boldsymbol{s}_0)\bigr]&=(\sigma_{11}^{(0)})^2-2\boldsymbol{c}^{(1)}_{1;n_1,n_2}{\!}^{\top}(\C^{(1)}_{n_1,n_2})^{-1}\boldsymbol{c}^{(0)}_{1;n_1,n_2}\\ &\quad+ \boldsymbol{c}^{(1)}_{1;n_1,n_2}{\!}^{\top}(\C^{(1)}_{n_1,n_2})^{-1}\C^{(0)}_{n_1,n_2} (\C^{(1)}_{n_1,n_2})^{-1}\boldsymbol{c}^{(1)}_{1;n_1,n_2}
\end{split}
\end{align}

and if the  true and misspecified models coincide then (\ref{mse_miss3}) simplifies for $a=0,1$ to
\begin{equation}\label{mse_miss4}
\var_{\mathcal{BC}^{(a)}}\bigl[\widehat{Z}^{\mathcal{BC}^{(a)}}_{1;n_1,n_2}(\ss_0)-Z_1(\boldsymbol{s}_0)\bigr]=(\sigma_{11}^{(a)})^2 -\boldsymbol{c}^{(a)}_{1;n_1,n_2}{\!}^{\top}(\C^{(a)}_{n_1,n_2})^{-1}\boldsymbol{c}^{(a)}_{1;n_1,n_2}.
\end{equation}
The following  Theorem follows directly from the arguments in \cite[Section~4.3]{Stein:1999} extended to the bivariate case. Hence the proof is omitted. We remark that these arguments indeed do not require collocated observation locations.

\begin{theorem}\label{theorem:kauf:23}
Let $\{ \ss_{1,1},\ldots,\ss_{1,n_1} \}$ and $\{ \ss_{2,1},\ldots,\ss_{2,n_2} \}$ be dense in $D$ as $n_1,n_2 \to \infty$.
For all $\boldsymbol{s}_0\in D$,
if $P(\mathcal{BC}^{(0)})\equiv P(\mathcal{BC}^{(1)})$ on the paths of $\bZ$ then:

\begin{enumerate}
  \item As $n_1,n_2\to \infty$ 
  \begin{equation}\label{kauf3_1} \frac{\var_{\mathcal{BC}^{(0)}}\bigl[ \widehat{Z}^{\mathcal{BC}^{(1)}}_{1;n_1,n_2}(\ss_0)-Z_1(\boldsymbol{s}_0)\bigr]}{\var_{\mathcal{BC}^{(0)}}\bigl[
\widehat{Z}^{\mathcal{BC}^{(0)}}_{1;n_1,n_2}(\ss_0)-Z_1(\boldsymbol{s}_0)\bigr]}{\,\longrightarrow\,}1.
        \end{equation}
   \item As $n_1,n_2\to \infty$  
     \begin{equation}\label{kauf3_2} \frac{\var_{\mathcal{BC}^{(1)}}\bigl[ \widehat{Z}^{\mathcal{BC}^{(1)}}_{1;n_1,n_2}(\ss_0)-Z_1(\boldsymbol{s}_0)\bigr]}{\var_{\mathcal{BC}^{(0)}}\bigl[
\widehat{Z}^{\mathcal{BC}^{(1)}}_{1;n_1,n_2}(\ss_0)-Z_1(\boldsymbol{s}_0)\bigr]}{\,\longrightarrow\,}1.
        \end{equation}
   \end{enumerate}
\end{theorem}
Then we can apply Theorem   \ref{theorem:kauf:23}
using the results on the equivalence of Gaussian measures given  in Section~\ref{SECTION:APPLICATION:MATERN:WENDLAND}
between two bivariate Mat{\'e}rn  models, two bivariate Generalized Wendland models and between a bivariate Mat{\'e}rn  and a Generalized Wendland model.

In particular, we apply  Theorem   \ref{theorem:kauf:23} when considering a bivariate Mat{\'e}rn and a bivariate Generalized Wendland model, probably the most interesting case.
With this   goal in mind, we consider the cases
$\mathcal{BC}^{(0)}= {\cal BM}_{\btheta^{(0)}}$ and  $\mathcal{BC}^{(1)}={\cal BW}_{\blambda^{(1)}}$
 defined in (\ref{eq:piip5}) and (\ref{eq:piip6}).

\begin{theorem}\label{kauf_3}
For given $\nu\geq1/2$ and $\kappa \geq 0$, consider $P({\cal BM}_{\btheta^{(0)}})$ and $P({\cal BW}_{\blambda^{(1)}})$. Let for simplicity $\alpha_{ij}^{(0)} = \alpha_{ij}$ and $\beta_{ij}^{(1)} = \beta_{ij}$ for $i,j=1,2$. 
Assume that $\nu = \kappa + 1/2$, $\mu > d+1/2+\kappa$ and that \eqref{cafu} holds. Let $\{ \ss_{1,1},\ldots,\ss_{1,n_1} \}$ and $\{ \ss_{2,1},\ldots,\ss_{2,n_2} \}$ be dense in $D$ as $n_1,n_2 \to \infty$. 
Then for $d=1, 2, 3$:
\begin{enumerate}
\item As $n_1,n_2\to \infty$ 
  \begin{equation}\label{miss1} \frac{\var_{{\cal BM}_{\btheta^{(0)}}}\bigl[ \widehat{Z}^{{\cal BW}_{\blambda^{(1)}}}_{1;n_1,n_2}(\ss_0)-Z_1(\boldsymbol{s}_0)\bigr]}{\var_{{\cal BM}_{\btheta^{(0)}}}\bigl[
\widehat{Z}^{{\cal BM}_{\btheta^{(0)}}}_{1;n_1,n_2}(\ss_0)-Z_1(\boldsymbol{s}_0)\bigr]}{\,\longrightarrow\,}1.
        \end{equation}
   \item As $n_1,n_2\to \infty$  
     \begin{equation}\label{miss2} \frac{\var_{{\cal BW}_{\blambda^{(1)}}}\bigl[ \widehat{Z}^{{\cal BW}_{\blambda^{(1)}}}_{1;n_1,n_2}(\ss_0)-Z_1(\boldsymbol{s}_0)\bigr]}{\var_{{\cal BM}_{\btheta^{(0)}}}\bigl[
\widehat{Z}^{{\cal BW}_{\blambda^{(1)}}}_{1;n_1,n_2}(\ss_0)-Z_1(\boldsymbol{s}_0)\bigr]}{\,\longrightarrow\,}1.
        \end{equation}
   \end{enumerate}
\end{theorem}

An important implication of \eqref{miss1} and \eqref{miss2} is that, if $\nu=\kappa+1/2$, $\mu > d+1/2+\kappa$, and under  condition~\eqref{cafu}, asymptotic cokriging prediction efficiency and asymptotically correct estimates of error variance are achieved using a bivariate Generalized Wendland model when the true model is bivariate Mat{\'e}rn.
This result has important practical implications, since the Generalized Wendland matrix covariance functions, unlike the Mat\'ern ones, are compactly supported. Hence, using a Generalized Wendland model provides important computational benefits, by enabling to exploit sparse matrix structures \cite{Furrer:2006,Kaufman:Schervish:Nychka:2008,bevilacqua2019estimation}, with a typically negligible loss of statistical accuracy if the true matrix covariance function is in the Mat\'ern class.

\section{Numerical Illustration}\label{sec:numill}
In this section we present a numerical illustration of the rates of convergence of the ratios~\eqref{miss1} and~\eqref{miss2}.
The mean square prediction error (MSPE) for kriging and cokriging can be interpreted as a statistic of the observation locations in relation to the prediction location, i.e., the MSPE essentially depends on the distance to the nearest observation location(s). Thus we work with regular grids specified as follows.
In one dimension, the observation locations for the primary variable (component of the multivariate random field) are $(k-1)/(n_1-1)$, $k=1,\dots,n_1$, with $n_1$ even.
For the secondary variable we select $(\ell-1)/(n_2-1)$ with $n_2=n_1,1.5n_1,3n_1$, $\ell=1,\ldots,n_2$. In two dimensions we take $n_1=n_x^2$ observation locations at $\bigl((k-1)/(n_x-1),(k'-1)/(n_x-1)\bigr)$, with $n_x$ even, $k,k'=1,\dots,n_x$. For the secondary variable we select a similar grid with $n_2=n_x^2, (1.5n_x)^2, (3n_x)^2$. 
Prediction for the first variable is at the center of the domain, i.e., $\s_0=0.5$ and $\s_0=(0.5,0.5)$, respectively.

We consider a bivariate Matérn model with
$\btheta^{(0)}= (\sigma^{(0)}_{11},\sigma^{(0)}_{22},\rho_{12},\nu,\alpha_{11},\alpha_{22},\alpha_{12})^\top=(1.2,1.1,$ $0.2,1/2,0.05,0.09,0.07)^\top$. %, for $\kappa=0,1$. 
In the first illustration we keep the same marginal variances and the same correlation parameter for the bivariate Generalized Wendland model with $\kappa=1$ and $\mu=5$. The range parameters are chosen according to the equivalence condition~\eqref{cafu} and yield for $\kappa=1$ the parameter vector   $\blambda^{(1)}=(1.2,1.1,0.2,1,0.297,$ $0.535, 0.416)^\top$.

\begin{figure}[h]
\includegraphics[width=\textwidth]{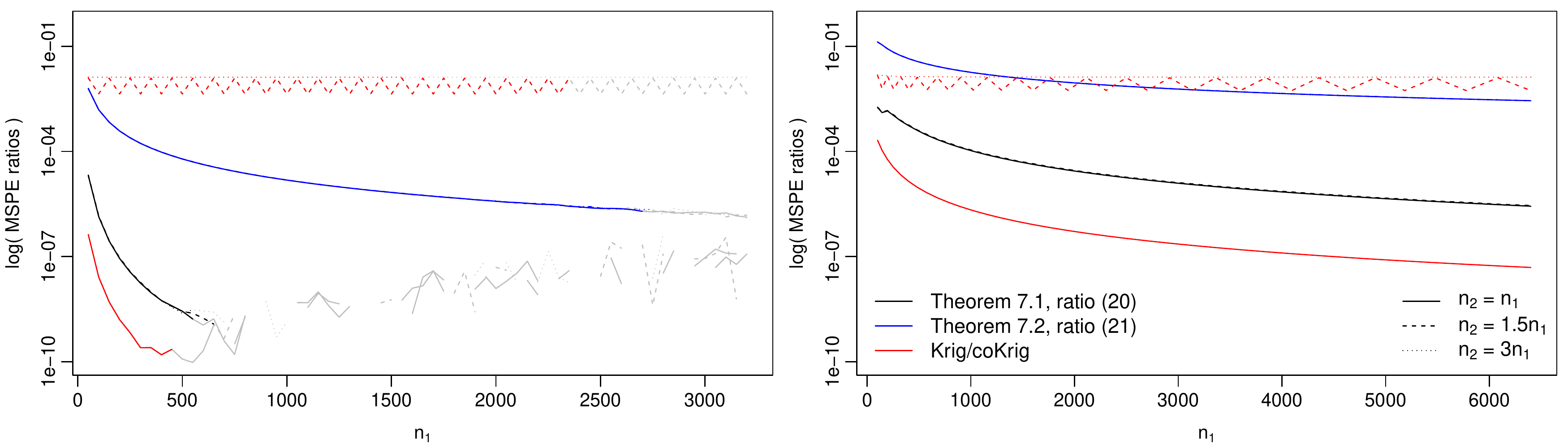}
\caption{Log of ratios~\eqref{miss1}, \eqref{miss2} and log of MSPE of  kriging versus cokriging  as a function of $n_1$ in one (left) and two (right) dimensions. Gray lines indicate numerical instabilities. \label{fig1}}
\end{figure}

Figure~\ref{fig1} illustrates the ratios~\eqref{miss1}, \eqref{miss2} and the ratio of MSPE of kriging versus cokriging in one and two dimensions. 
The convergence of the ratios is fast, and numerical instabilities are observed in one dimension for quite small $n_1$. 
Except for the kriging/cokriging ratio, increasing the number of location points for the secondary variable has only a very minor effect and can hardly be distinguished visually. The saw-tooth shape of the dashed red line is due to the alternating even/odd number of location observations $n_2$. 
For a fixed $n_1$, there is of course a nonlinear relation between the ratios and where we exactly place the point to predict within the observed grid. In one dimension, the log-ratio can be reduced by roughly a factor of two if we move the prediction location from $0.5$ towards the nearest right observation location $n_1/(2n_1-2)$. 
The left panel of Figure~\ref{fig2} illustrates the log ratios as a function of the grid spacing and emphasizes again that the MSPE is essentially driven by the locations of the nearby observations. The convergence rate for~\eqref{miss1} is slightly higher compared to \eqref{miss2} but equivalent to the ratio kriging versus cokriging in case of $n_1=n_2$.

\begin{figure}[h]
\includegraphics[width=\textwidth]{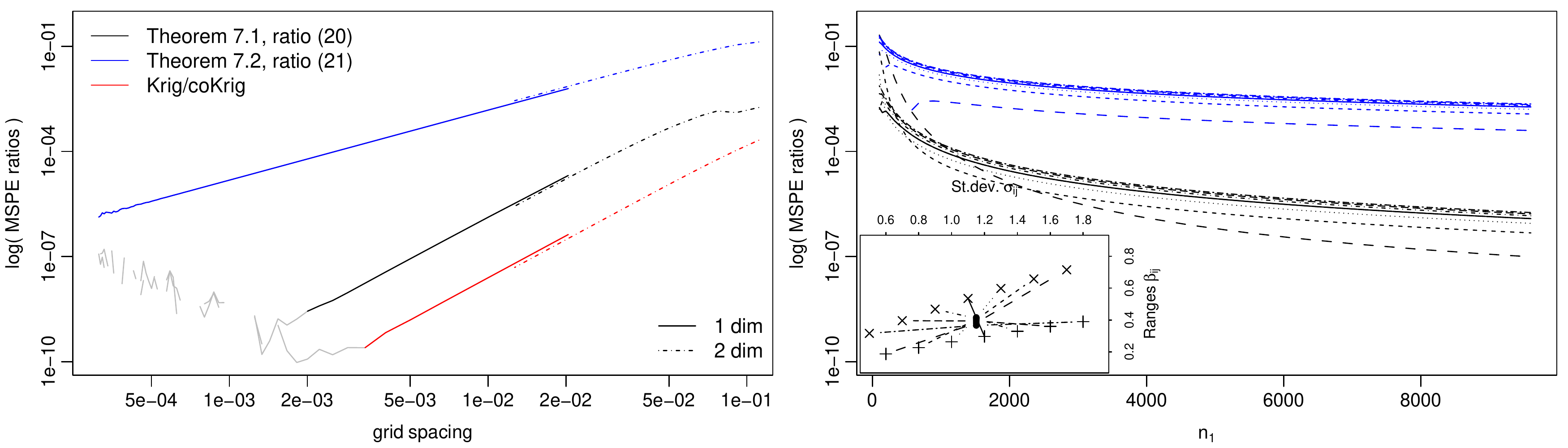}
\caption{Left panel: Log of ratios~\eqref{miss1}, \eqref{miss2} and log of MSPE of kriging versus cokriging in one (solid) and two dimensions (dash-dotted) as a function of grid spacing. Transparent lines indicate numerical instabilities. Right panel: Log of ratios~\eqref{miss1}, \eqref{miss2} for different parameter settings of the bivariate Wendland model. The inset figure shows the different configurations with varying $\sigma_{11}$ ($+$) and $\sigma_{22}$ ($\times$) and induced range parameters $\beta_{11}$, $\beta_{12}$ and $\beta_{22}$. ($\beta_{12}$ is plotted at fixed $x$-axis value $1.15$.) In both panels $n_2 = n_1$. \label{fig2}}
\end{figure}

To study the effect of different ranges we modify the variance parameters of the Generalized Wendland by $\bigl(\sigma^{(1)}_{11}+\delta,\sigma^{(1)}_{22}-\delta\bigr)$, $\delta=-0.6,-0.4,\dots,0.6$. 
The range parameters are updated according to~\eqref{cafu}, leading to shorter ranges for smaller standard deviations. 
The right panel of Figure~\ref{fig2} shows that the ratios are quite stable with respect to different ranges. Increasing the range parameter of the secondary variable reduces the ratio. 
Hence, it is possible to choose a range parameter $\beta_{ij}$ tailored to available computing and memory amount with a bearable cost in terms of MSPE.

Note that for other values of the smoothness parameter $\nu$ the rates themselves change but the conclusions remain the same. Similarly, scaling the ranges of the covariance parameters of the Matérn model has no effect on the asymptotic results as the scaling is essentially equivalent to adapting the number of observation points.

%\section*{Acknowledgments}

\appendix

\section{Proofs for Section \ref{SECTION:GENERAL:CONDITIONS:EQUIVALENCE}}

\begin{proof}[{\bf Proof of Theorem \ref{theorem:sufficient:representation:b}}]

Let $\bmu \in \IR^d$. We consider the integral operator $V$ on $\cW_{D}(\bF^{(0)})$ defined by 
\[
(V \bof)(\bmu)
=
\int_{\IR^d} 
  \bB( \bmu , \blambda )  \bF^{(0)}(\blambda)
\bof(\blambda)
d \blambda.
\]
We note that \eqref{eq:the:finite:integral} in concert with Cauchy-Schwarz inequality imply
\begin{equation} \label{eq:check:Vf:defined}
\int_{\IR^d}
\| \B(\bmu , \blambda ) \|
\|\F^{(0)}(\blambda)\|
\| \bof(\blambda)\|
d \blambda 
\leq 
\sqrt{
\int_{\IR^d}
\| \B(\bmu , \blambda ) \|^2
\|\F^{(0)}(\blambda)\|
d \blambda 
} 
\sqrt{
\int_{\IR^d}
\|\F^{(0)}(\blambda)\|
\| \bof(\blambda)\|^2
d \blambda. 
}
\end{equation}
The first integral in \eqref{eq:check:Vf:defined} is finite for almost all $\bmu \in \IR^d$ from \eqref{eq:the:finite:integral}. The second integral in \eqref{eq:check:Vf:defined} is smaller than a constant times $\int_{ \IR^d }  \bar{\bof}(\blambda)^\top \bF^{(0)}(\blambda) \bof(\blambda) d \blambda < \infty$ from Condition \ref{cond:phi:zero}. Hence, $V \bof$ is well-defined as a function from $\IR^d$ to $\IC^p$.

Let us now check that $V \bof$ belongs to $\cW_{D}(\bF^{(0)})$ when $\bof$ belongs to $\cW_{D}(\bF^{(0)})$. We use $f \in \cW_{D}(\bF^{(0)})$ and repeatedly apply Cauchy-Schwarz inequality, so that, for a finite constant $c$,
\begin{align*}
\| V \bof &\|^2_{\cW_{D}(\bF^{(0)})} 
 = 
\int_{\IR^d} 
\int_{\IR^d} 
\int_{\IR^d} 
\bar{\bof}(\blambda)^\top 
\bF^{(0)}(\blambda)
\bar{\bB}(\blambda,\t)
\bF^{(0)}(\t)
\bB(\t,\bmu)  \bF^{(0)}(\bmu)
\bof(\bmu)
d \t d \blambda d \bmu \\
&\leq  
\int_{\IR^d} 
\int_{\IR^d} 
\int_{\IR^d}
\| \bof(\blambda)\|\, 
\|
\bB(\blambda,\t)
\|\,
\| \bF^{(0)}(\blambda)
\|\,
\|
\bF^{(0)}(\t)
\|\,
\| \bB(\bmu,\t)
\|
\\
& \qquad
\| \bF^{(0)}(\bmu)
\|\,
\| \bof(\bmu)\|
d \t d\blambda d\bmu
\\
&= 
\int_{\IR^d} 
d \t
\|\bF^{(0)}(\t)\|
\left(
\int_{\IR^d}
d \blambda 
\|\bF^{(0)}(\blambda)\|\, 
\| \bof(\blambda)\| \|
\bB(\blambda,\t)
\|
\right)
\\
&\qquad
\left(
\int_{\IR^d}
d \bmu 
\|\bF^{(0)}(\bmu)\|\,
\| \bof(\bmu)\|\,\|
\bB(\bmu,\t)
\|
\right)
\\
&\leq 
\int_{\IR^d} 
d \t
\|\bF^{(0)}(\t)\|
\sqrt{
\left(
\int_{\IR^d}
d \blambda 
\|\bF^{(0)}(\blambda)\|\,
\| \bof(\blambda)\|^2
\right)
\left(
\int_{\IR^d}
d \blambda 
\|\bF^{(0)}(\blambda)\|\,
\|\bB(\blambda,\t)\|^2
\right)
}
\\
&\qquad
\sqrt{
\left(
\int_{\IR^d}
d \bmu 
\|\bF^{(0)}(\bmu)\|\,
\| \bof(\bmu)\|^2
\right)
\left(
\int_{\IR^d}
d \bmu
\|\bF^{(0)}(\bmu)\|\,
\|\bB(\bmu,\t)\|^2
\right)
}\\
& \leq 
c
\left(
\int_{\IR^d}
d \blambda 
\bar{ \bof}(\blambda)^\top 
\bF^{(0)}(\blambda)
 \bof(\blambda)
\right)
\int_{\IR^d} 
d \t
\|\bF^{(0)}(\t)\|
\\
& \qquad
\sqrt{
\int_{\IR^d}
d \blambda 
\|\bF^{(0)}(\blambda)\|\,
\|\bB(\blambda,\t)\|^2
}
\sqrt{
\int_{\IR^d}
d \bmu
\|\bF^{(0)}(\bmu)\|\,
\|\bB(\bmu,\t)\|^2
}
\\
 & = 
 c \| \bof \|^2_{\cW_{D}(\bF^{(0)})}
 \int_{\IR^d} 
d \t
\|\bF^{(0)}(\t)\|
\int_{\IR^d}
d \blambda 
\|\bF^{(0)}(\blambda)\|\, 
\|\bB(\blambda,\t)\|^2
% \\
~~ <  + \infty. \qquad\qquad
\end{align*}
In the strict inequality below, we have used \eqref{eq:the:finite:integral}. In the second to last ``$\leq$'' we have used
Condition \ref{cond:phi:zero}. 

Hence $V$ maps $\cW_{D}(\bF^{(0)})$ to $\cW_{D}(\bF^{(0)})$. 
Let us check that $V$ is Hermitian. For any $\bof_1 , \bof_2 \in \cW_{D}(\bF^{(0)})$, we have
\begin{align*}
( V \bof_1 , \bof_2 )_{\cW_{D}(\bF^{(0)})}
= &
\int_{\IR^d}
d \blambda
\bar{\bof_2}(\blambda)^\top
\bF^{(0)}(\blambda)
\left(
\int_{\IR^d}
d \bmu 
\bB(\blambda,\bmu)
\bF^{(0)}(\bmu)
\bof_1(\bmu)
\right)
\\
= &
\int_{\IR^d}
\int_{\IR^d}
d \blambda 
d \bmu 
\bar{\bof_2}(\blambda)^\top
\bF^{(0)}(\blambda)
\bB(\blambda,\bmu)
\bF^{(0)}(\bmu)
\bof_1(\bmu)
\\
= &
\int_{\IR^d}
\int_{\IR^d}
d \blambda 
d \bmu 
\bar{\bof_1}(\bmu)^\top
\bF^{(0)}(\bmu)
\bB(\bmu,\blambda)
\bF^{(0)}(\blambda)
\bof_2(\blambda)
\\
= &
( V \bof_2 , \bof_1 )_{\cW_{D}(\bF^{(0)})},
\end{align*}
where we have used $\bar{\bB}(\blambda , \bmu)^\top = \bB(\bmu,\blambda)$.

Let us now show that $V$ is an Hilbert--Schmidt operator. Let $(\bphi_k)_{k \in \IN}$ be an orthonormal basis of $\cW_{D}(\bF^{(0)})$. Let, for $n \in \mathbb{N}$, $\blambda , \bmu \in \IR^d$,
\[
\overline{\bB_n( \blambda , \bmu)}
=
\sum_{a=1}^n
\left(
\int_{\IR^d}
\bB( \blambda , \bz )
\bF^{(0)}(\bz)
\bphi_a(\bz)
d \bz
\right)
\bphi_a(\bmu)^\top.
\]
Let also $\overline{\bE_n( \blambda , \bmu)} = \overline{\bB_n( \blambda , \bmu)} - \overline{\bB( \blambda , \bmu)}$.
Then, we have 
\begin{align}
\trace 
\bigg(
\int_{\IR^d}&
\int_{\IR^d}
(\bF^{(0)})^{1/2}(\blambda)
\overline{\bE_n( \blambda , \bmu)}
\bF^{(0)}(\bmu)
\bE_n( \blambda , \bmu)^\top
(\bF^{(0)})^{1/2}(\blambda)
d \blambda 
d \bmu
\biggr)
\label{eq:the:trace}
\\
& \leq 
p
\int_{\IR^d}
d \blambda 
\| \bF^{(0)}(\blambda) \|
\int_{\IR^d}
d \bmu
\Bnorml
\overline{\bE_n( \blambda , \bmu)}
\bF^{(0)}(\bmu)
\bE_n( \blambda , \bmu)^\top
\Bnormr
\notag
\\
& \leq 
p
\sum_{a=1}^p
\int_{\IR^d}
d \blambda 
\| \bF^{(0)}(\blambda) \|
\int_{\IR^d}
d \bmu
\left(
\bfe_a^\top
\overline{\bE_n( \blambda , \bmu)}
\right)
\bF^{(0)}(\bmu)
\left(
\bfe_a^\top
\bE_n( \blambda , \bmu)
\right)^\top.
\label{eq:inner:most:int}
\end{align}
We note that $\bfe_a^\top
\overline{\bB_n( \blambda , \bmu)}$ is the orthogonal projection in $\cW_{D}(\bF^{(0)})$ of the row $a$ of $ \bmu \mapsto \overline{\bB(\blambda,\bmu)}$ on the linear space spanned by $\bphi_1^\top,\ldots,\bphi_n^\top$. For almost all $\blambda \in \IR^d$, the norm of this row in $\cW_{D}(\bF^{(0)})$ is finite from \eqref{eq:the:finite:integral}. Hence, for almost all $\blambda \in \IR^d$, the inner  most integral in \eqref{eq:inner:most:int} goes to zero as $n \to \infty$. Furthermore, this inner most integral is bounded by 
\[
\int_{\IR^d}
d \bmu
\left(
\bfe_a^\top
\overline{\bB( \blambda , \bmu)}
\right)
\bF^{(0)}(\bmu)
\left(
\bfe_a^\top
\bB( \blambda , \bmu)
\right)^\top,
\]
which satisfies
\[
\int_{\IR^d}
d \blambda 
\| \bF^{(0)}(\blambda) \|
\int_{\IR^d}
d \bmu
\left(
\bfe_a^\top
\overline{\bB( \blambda , \bmu)}
\right)
\bF^{(0)}(\bmu)
\left(
\bfe_a^\top
\bB( \blambda , \bmu)
\right)^\top
< \infty
\]
from \eqref{eq:the:finite:integral}. Hence, by the dominated convergence theorem, \eqref{eq:the:trace} goes to zero as $n \to \infty$. Now, consider the application 
\[
\A , \C
\mapsto
\trace 
\left(
\int_{\IR^d}
\int_{\IR^d}
(\bF^{(0)})^{1/2}(\blambda)
\A( \blambda , \bmu)
\bF^{(0)}(\bmu)
\overline{\C( \blambda , \bmu)}^\top
(\bF^{(0)})^{1/2}(\blambda)
d \blambda 
d \bmu
\right),
\]
for functions $\A , \C$ from $\IR^d \times \IR^d$ to $\IC^{p^2}$ satisfying \eqref{eq:the:finite:integral} with $\bB$ replaced by $\A$ or $\C$ there. One can check that this application is a scalar product. Hence, using the triangle inequality, it follows that 
\[
\trace 
\left(
\int_{\IR^d}
\int_{\IR^d}
(\bF^{(0)})^{1/2}(\blambda)
\overline{\bB_n( \blambda , \bmu)}
\bF^{(0)}(\bmu)
\bB_n( \blambda , \bmu)^\top
(\bF^{(0)})^{1/2}(\blambda)
d \blambda 
d \bmu
\right)
\]
is bounded as $n \to \infty$. This bounded quantity is equal to, using the orthogonality of $\bphi_1,\ldots,\bphi_n$,
\begin{align*}
\sum_{i,j=1}^n &
\trace 
\Bigg(
\int_{\IR^d}
\int_{\IR^d}
(\bF^{(0)})^{1/2}(\blambda)
\left(
\int_{\IR^d}
\bB( \blambda , \bz )
\bF^{(0)}(\bz)
\bphi_i(\bz)
d \bz
\right)
\bphi_i(\bmu)^\top
\bF^{(0)}(\bmu)
\\
& \quad
\overline{\bphi_j(\bmu)}
\overline{
\left(
\int_{\IR^d}
\bB( \blambda , \bz )
\bF^{(0)}(\bz)
\bphi_j(\bz)
 d \bz
\right)
}^\top
(\bF^{(0)})^{1/2}(\blambda)
d \blambda 
d \bmu
\Bigg)
\\
&= 
\sum_{i=1}^n
\trace 
\Bigg(
\int_{\IR^d}
(\bF^{(0)})^{1/2}(\blambda)
\left(
\int_{\IR^d}
\bB( \blambda , \bz )
\bF^{(0)}(\bz)
\bphi_i(\bz)
 d \bz
\right)
\\
&\quad
\overline{
\left(
\int_{\IR^d}
\bB( \blambda , \bz )
\bF^{(0)}(\bz)
\bphi_i(\bz)
d \bz
\right)
}^\top
(\bF^{(0)})^{1/2}(\blambda)
d \blambda 
\Bigg)
\\
& = 
\sum_{i=1}^n
\trace 
\Bigg(
\int_{\IR^d}
\overline{
\left(
\int_{\IR^d}
\bB( \blambda , \bz )
\bF^{(0)}(\bz)
\bphi_i(\bz)
 d \bz
\right)
}^\top
(\bF^{(0)})^{1/2}(\blambda)
\\
& \qquad
(\bF^{(0)})^{1/2}(\blambda)
\left(
\int_{\IR^d}
\bB( \blambda , \bz )
\bF^{(0)}(\bz)
\bphi_i(\bz) 
d \bz
\right)
d \blambda 
\Bigg)
\\
& = 
\sum_{i=1}^n
\int_{\IR^d}
\overline{
\left(
\int_{\IR^d}
\bB( \blambda , \bz )
\bF^{(0)}(\bz)
\bphi_i(\bz)
 d \bz
\right)
}^\top
\bF^{(0)}(\blambda)
\left(
\int_{\IR^d}
\bB( \blambda , \bz )
\bF^{(0)}(\bz)
\bphi_i(\bz) 
d \bz
\right)
d \blambda 
\\
 & = 
\sum_{i=1}^n
\| V \bphi_i \|_{\cW_{D}(\bF^{(0)})}^2.
\end{align*}
This implies that
\[
\sum_{i=1}^\infty
\| V \bphi_i \|_{\cW_{D}(\bF^{(0)})}^2
< \infty
\]
and thus $V$ is Hilbert--Schmidt.

Hence, there exists a sequence $(\bg_k)_{k \in \mathbb{N}}$ of eigenfunctions of $V$. For $k \in \IN$, we let $\bg_k = ( g_{k,1} , \ldots, g_{k,p} )^\top$ from $\IR^d$ to $\IC^p$ and we remark that we have $( \bg_k , \bg_j )_{\cW_{D}(\bF^{(0)})} = \delta_{k,j}$ for $k,j \in \IN$. We let $(\lambda_k)_{k \in \IN}$ be the corresponding sequence of eigenvalues of $V$, such that we have $V \bg_k  = \lambda_k \bg_k$ for $k \in \IN$ and $\sum_{k=1}^{\infty} \lambda_k^2 <  \infty$. 

Let $k, j \in \IN$ be fixed. By definition of $\cW_{D}( \bF^{(0)} )$, there exists a sequence $( \bphi_{k,n} )_{n \in \IN}$ such that $\bphi_{k,n} : D \to \IC^p$ for $n \in \IN$ and such that, with $\bu_{k,n} = ( u_{k,n,1} , \ldots, u_{k,n,p})^\top$ from $\IR^d$ to $\IC^p$ defined by
\[
u_{k,n,i}( \blambda ) = \int_{D}  e^{- \i \blambda^\top \t} \phi_{k,n,i}( \t) d \t,
\]
for $i=1,\ldots,p$ and $\blambda \in \IR^d$, we have $\bu_{k,n} \to \bg_{k}$ in $\cW_{D}( \bF^{(0)} )$. There also exists a sequence $( \bphi_{j,n} )_{n \in \IN}$ that is defined similarly for $\bg_j$ instead of $\bg_k$.

We have, using \eqref{eq:assumption:Rdtwo:minus:Rone},
\begin{align} 
& 
\int_{D} \int_{D}
\bar{\bphi}_{k,n}(\t)^\top
\bR^{(1)}(\t-\s)
\bphi_{j,n}(\s)
d \t
d \s
-
\int_{D} \int_{D}
\bar{\bphi}_{k,n}(\t)^\top
\bR^{(0)}(\t-\s)
\bphi_{j,n}(\s)
d \t
d \s
\label{eq:double:int:Rtwo:minus:Rone:one}
\\
&\quad =
\int_{D} \int_{D} \int_{\IR^d} \int_{\IR^d}
e^{ - \i \blambda^\top \t + \i \bmu^\top \s }
\bar{\bphi}_{k,n}(\t)^\top
\bF^{(0)}( \blambda )
\bB( \blambda , \bmu )
\bF^{(0)}(\bmu)
\bphi_{j,n}(\s)
d \t
d \s
d\blambda
d\bmu
\nonumber
\\
& \quad = 
\int_{\IR^d} \int_{\IR^d}
\bar{\bu}_{k,n}(\blambda)^\top
\bF^{(0)}( \blambda )
\bB( \blambda , \bmu )
\bF^{(0)}(\bmu)
\bu_{j,n}(\bmu)
d\blambda d\bmu.
\label{eq:double:int:Rtwo:minus:Rone:two}
\end{align}
Let us find the limit of the two terms in \eqref{eq:double:int:Rtwo:minus:Rone:one} as $n \to \infty$. We have
\begin{align*}
\int_{D} \int_{D}
\bar{\bphi}_{k,n}(\t)^\top
\bR^{(1)}(\t-\s)
\bphi_{j,n}(\s)
d \t
d \s
& = 
\int_{D} \int_{D} \int_{\IR^d}
e^{\i \blambda^\top (\t-\s)}
\bar{\bphi}_{k,n}(\t)^\top
\bF^{(1)}(\blambda)
\bphi_{j,n}(\s)
d \t
d \s
d\blambda
\\
& =
\int_{\IR^d}
\bar{\bu}_{k,n}(\blambda)^\top
\bF^{(1)}(\blambda)
\bu_{k,n}(\blambda)
d\blambda \\
& \to_{n \to \infty}
\int_{\IR^d}
\bar{\bg}_{k}(\blambda)^\top
\bF^{(1)}(\blambda)
\bg_{j}(\blambda)
d\blambda
\end{align*}
using Lemma \ref{lemma:fdeux:smaller:constant:Fun} and the triangle inequality.
Similarly we have
\[
\int_{D} \int_{D}
\bar{\bphi}_{k,n}(\t)^\top
\bR^{(0)}(\t-\s)
\bphi_{j,n}(\s)
d \t
d \s
\to_{n \to \infty}
\int_{\IR^d}
\bar{\bg}_{k}(\blambda)^\top
\bF^{(0)}(\blambda)
\bg_{j}(\blambda)
d\blambda.
\]
Let us find the limit of $\eqref{eq:double:int:Rtwo:minus:Rone:two}$ as $n \to \infty$.
We have, with a finite constant $c$,
\begin{align*}
&
\int_{\IR^d} \int_{\IR^d}
\bar{\bu}_{k,n}(\blambda)^\top
\bF^{(0)}( \blambda )
\bB( \blambda , \bmu )
\bF^{(0)}(\bmu)
\bu_{j,n}(\bmu)
d\blambda d\bmu 
\\
& ~ ~
-
\int_{\IR^d} \int_{\IR^d}
\bar{\bg}_{k}(\blambda)^\top
\bF^{(0)}( \blambda )
\bB( \blambda , \bmu )
\bF^{(0)}(\bmu)
\bu_{j,n}(\bmu)
d\blambda d\bmu 
\\
& = 
\int_{\IR^d} \int_{\IR^d}
\left[\bar{\bu}_{k,n}(\blambda) - \bar{\bg}_{k}(\blambda) \right]^\top
\bF^{(0)}( \blambda )
\bB( \blambda , \bmu )
\bF^{(0)}(\bmu)
\bu_{j,n}(\bmu)
d\blambda d\bmu \hspace*{2.8cm}
%%%%%%%#######################
%%%%%%%NON-robust page break mea culpa
%\\
\end{align*}
\begin{align*}
& \leq 
c
\sqrt{
\int_{\IR^d} \int_{\IR^d}
\|\bF^{(0)}( \blambda )\|
\|\bB( \blambda , \bmu )\|^2
\|\bF^{(0)}(\bmu)\|
d\blambda d\bmu
}
\\
& ~ ~ ~ 
 \sqrt{
\int_{\IR^d} 
\int_{\IR^d} 
\|\bF^{(0)}( \blambda )\|
\| \bu_{k,n}(\blambda) - \bg_{k}(\blambda) \|^2
\|\bF^{(0)}( \bmu )\|
\|   \bu_{j,n}(\bmu) \|^2
d \bmu 
d\blambda 
}
\\
& =
c
\sqrt{
\int_{\IR^d} \int_{\IR^d}
\|\bF^{(0)}( \blambda )\|
\|\bB( \blambda , \bmu )\|^2
\|\bF^{(0)}(\bmu)\|
d\blambda d\bmu
}
 \sqrt{
\int_{\IR^d} 
\|\bF^{(0)}( \blambda )\|
\| \bu_{k,n}(\blambda) - \bg_{k}(\blambda) \|^2
d\blambda 
}
\\
&
~ ~ ~
 \sqrt{
\int_{\IR^d} 
\|\bF^{(0)}( \bmu )\|
\|   \bu_{j,n}(\bmu) \|^2
d \bmu 
}
\\
&
\to_{n \to \infty} 0,
\end{align*}
where we have used the Cauchy--Schwarz inequality and the fact that $\bu_{k,n}$ converges to $\bg_k$ in $\cW_{D}(\bF^{(0)})$, together with Condition \ref{cond:phi:zero} and \eqref{eq:the:finite:integral}.

We show similarly
\begin{align*}
&
\int_{\IR^d} \int_{\IR^d}
\bar{\bg}_{k}(\blambda)^\top
\bF^{(0)}( \blambda )
\bB( \blambda , \bmu )
\bF^{(0)}(\bmu)
\bu_{j,n}(\bmu)
d\blambda d\bmu 
\\
& ~ ~
-
\int_{\IR^d} \int_{\IR^d}
\bar{\bg}_{k}(\blambda)^\top
\bF^{(0)}( \blambda )
\bB( \blambda , \bmu )
\bF^{(0)}(\bmu)
\bg_{j}(\bmu)
d\blambda d\bmu \quad
 \to_{n \to \infty} 0.
\end{align*}
Hence, \eqref{eq:double:int:Rtwo:minus:Rone:two} converges to
\[
\int_{\IR^d} \int_{\IR^d}
\bar{\bg}_{k}(\blambda)^\top
\bF^{(0)}( \blambda )
\bB( \blambda , \bmu )
\bF^{(0)}(\bmu)
\bg_{j}(\bmu)
d\blambda d\bmu
\]
as $n \to \infty$. Thus, from \eqref{eq:double:int:Rtwo:minus:Rone:one} and \eqref{eq:double:int:Rtwo:minus:Rone:two}, we obtain
\begin{equation} \label{eq:gkgjF2:minus:gKgjF1}
(\bg_k , \bg_j)_{\cW_{D}(\bF^{(1)})}
-
(\bg_k , \bg_j)_{\cW_{D}(\bF^{(0)})}
=
\lambda_k \delta_{jk}.
\end{equation}
Let $A$ be the operator from $\cW_{D} (\bF^{(0)})$ to $\cW_{D} (\bF^{(1)})$ defined by
\[
A \bphi = \bphi ~ ~ \in \cW_{D} (\bF^{(1)}).
\]
From Lemma \ref{lemma:fdeux:smaller:constant:Fun}, $A$
is well-defined and bounded. Let $A^\star$ from  $\cW_{D} (\bF^{(1)})$ to $\cW_{D} (\bF^{(0)})$ be the adjoint operator to $A$.  
We remark that we have for $\bphi_1,\bphi_2 \in \cW_{D} (\bF^{(0)})$,
\begin{equation} \label{eq:AstarA}
( A^\star A \bphi_1 , \bphi_2)_{\cW_{D} (\bF^{(0)})}
=
(  A \bphi_1 , A \bphi_2)_{\cW_{D} (\bF^{(1)})}
=
( \bphi_1 , \bphi_2)_{\cW_{D} (\bF^{(1)})}.
\end{equation}
Consider the operator $\Delta = I - A^\star A $ from $\cW_{D} (\bF^{(0)})$ to $\cW_{D} (\bF^{(0)})$ where $I$ is the identity operator. Then from \eqref{eq:AstarA} we have, for $\bphi_1,\bphi_2 \in \cW_{D} (\bF^{(0)})$,
\begin{equation} \label{eq:Delta:phi:scalaire:phi}
( \Delta \bphi_1 , \bphi_2)_{\cW_{D} (\bF^{(0)})}
=
( \bphi_1 , \bphi_2)_{\cW_{D} (\bF^{(0)})}
-
( \bphi_1 , \bphi_2)_{\cW_{D} (\bF^{(1)})}.
\end{equation}
Let now $\bphi_1,\ldots,\bphi_n$ be any orthonormal functions in $\cW_{D} (\bF^{(0)})$.
From \eqref{eq:gkgjF2:minus:gKgjF1} and \eqref{eq:Delta:phi:scalaire:phi}, we have, using Bessel's inequality and Parseval's identity,
\begin{align*}
+\infty
&
>
\sum_{k=1}^{\infty} \lambda_k^2
%\\& 
= 
\sum_{k,j=1}^{\infty}
\left(
( \Delta \bg_k , \bg_j)_{\cW_{D} (\bF^{(0)})}
\right)^2
%\\ &
=
\sum_{k=1}^{\infty}
\left(
\| \Delta \bg_k \|_{\cW_{D} (\bF^{(0)})}
\right)^2
\\
& \geq 
\sum_{k=1}^{\infty}
\sum_{j=1}^{n}
\left(
( \Delta \bg_k , \bphi_j)_{\cW_{D} (\bF^{(0)})}
\right)^2
%\\ & 
= 
\sum_{k=1}^{\infty}
\sum_{j=1}^{n}
\left(
(  \bg_k , \Delta \bphi_j)_{\cW_{D} (\bF^{(0)})}
\right)^2
\\
& =
 \sum_{j=1}^{n}
\left(
\| \Delta  \bphi_j \|_{\cW_{D} (\bF^{(0)})}
\right)^2
% \\ & 
\geq 
 \sum_{k,j=1}^{n}
\left(
( \bphi_k ,  \Delta  \bphi_j )_{\cW_{D} (\bF^{(0)})}
\right)^2.
\end{align*}
Hence, from \eqref{eq:Delta:phi:scalaire:phi} we have
\begin{equation} \label{eq:bound:diff:scalar:product:phi}
\sum_{k,j=1}^{n}
\left(
( \bphi_k ,   \bphi_j )_{\cW_{D} (\bF^{(0)})}
-
( \bphi_k ,   \bphi_j )_{\cW_{D} (\bF^{(1)})}
\right)^2
\leq 
\sum_{k=1}^{\infty} \lambda_k^2
< + \infty.
\end{equation}
Let for $a=0,1$, $\cB^{(a)}$ be the operator on $L^{2,p}_{D}$ defined by $ \cB^{(a)}(\bof)(\t)  = \int_{D}  \bR^{(a)}(\t - \bu) \bof(\bu) d \bu$.
Let $(\bh_{k})_{k \in \mathbb{N}}$ be the orthonormal basis of $ L^{2,p}_{D}$ composed of the eigenfunctions of $\cB^{(0)}$, with eigenvalues $(\rho_k)_{k \in \mathbb{N}}$ (the existence can be proved as for the proof that $V$ is Hilbert-Schmidt above). Let, for $k \in \mathbb{N}$, $\bphi_k = (\phi_{k,1},\ldots,\phi_{k,p})$ with $\phi_{k,i}(\blambda) =  \int_{D} h_{k,i}(\t) e^{ - \i \blambda^\top \t} d\t  $ for $i=1,\ldots,p$. Then $\bphi_k \in \cW_{D}(\bF^{(0)})$ for $k \in \mathbb{N}$ and we have, for $k,j \in \mathbb{N}$,
\begin{align} \label{eq:from:WF2:to:B2}
(\bphi_k , \bphi_j)_{\cW_{D}(\bF^{(1)})}
& =
\int_{D}
\int_{D}
\int_{\IR^d}
\bar{\bh}_{k}(\t)^\top e^{  \i \blambda^\top \t}
\bF^{(1)}(\blambda)
\bh_{j}(\bu) e^{- \i \blambda^\top \bu}
d \t d \bu d\blambda
\notag
\\
& = \int_{D}
\int_{D}
\bar{\bh}_{k}(\t)^\top
\bR^{(1)}(\t-\bu)
\bh_j(\bu).
\end{align}
Similarly,
\begin{align} \label{eq:from:WF1:to:B1}
(\bphi_k , \bphi_j)_{\cW_{D}(\bF^{(0)})}
& =
 \int_{D}
\int_{D}
\bar{\bh}_{k}(\t)^\top
\bR^{(0)}(\t-\bu)
\bh_j(\bu).
\end{align}

In particular $(\bphi_k , \bphi_j)_{\cW_{D}(\bF^{(0)})} = \delta_{k,j} \rho_k$.
Since the $(\bphi_k)_{k \in \mathbb{N}}$ are not almost surely equal to zero, it follows that $\rho_k >0$ for $k \in \mathbb{N}$ from Condition \ref{cond:phi:zero}.

Hence,  from \eqref{eq:bound:diff:scalar:product:phi}, \eqref{eq:from:WF2:to:B2} and \eqref{eq:from:WF1:to:B1}, we obtain, with $\tilde{\bphi}_k = \bphi_k / \rho_k^{1/2}$,
\begin{align*}
\infty  
> \sum_{k=1}^{\infty} \lambda_k^2
% \\ & 
& \geq 
\sum_{k,j=1}^{n}
\left(
( \tilde{\bphi}_k ,   \tilde{\bphi}_j )_{\cW_{D} (\bF^{(0)})}
-
( \tilde{\bphi}_k ,   \tilde{\bphi}_j )_{\cW_{D} (\bF^{(1)})}
\right)^2
\\
& = 
\sum_{k,j=1}^{n}
\left(
\frac{1}{\sqrt{\rho_k}}
\frac{1}{\sqrt{\rho_j}}
( \bphi_k ,   \bphi_j )_{\cW_{D} (\bF^{(0)})}
-
\frac{1}{\sqrt{\rho_k}}
\frac{1}{\sqrt{\rho_j}}
( \bphi_k ,   \bphi_j )_{\cW_{D} (\bF^{(1)})}
\right)^2
\\
& =
 \sum_{k,j=1}^{n}
\left(
\frac{1}{\sqrt{\rho_k}}
\frac{1}{\sqrt{\rho_j}}
 (   \bh_k , \cB^{(0)} \bh_j   )_{L^{2,p}_{D}} 
-
\frac{1}{\sqrt{\rho_k}}
\frac{1}{\sqrt{\rho_j}}
 (   \bh_k , \cB^{(1)} \bh_j   )_{L^{2,p}_{D}} 
\right)^2.
\end{align*}
Let now $\bof \in L^{2,p}_{D}$. We can write $\bof = \sum_{i=1}^\infty \alpha_i \bh_i$ with $\sum_{i=1}^\infty \alpha_i^2  < \infty$. If $(\bof , \cB^{(0)} \bof)_{L^{2,p}_{D}} = 1$, we have $\sum_{i=1}^\infty \alpha_i^2 \rho_i = 1$ and
\begin{align*}
(\bof , &\cB^{(1)} \bof)_{L^{2,p}_{D}}
-
1
 \leq
\left|  
(\bof , \cB^{(0)} \bof)_{L^{2,p}_{D}}
-
(\bof , \cB^{(1)} \bof)_{L^{2,p}_{D}}
\right|
\\
& \leq 
\sum_{i,j=1}^\infty
|\alpha_i \alpha _j  |
\left|
(\bh_i , \cB^{(0)} \bh_j)_{L^{2,p}_{D}}
-
(\bh_i , \cB^{(1)} \bh_j)_{L^{2,p}_{D}}
\right|
\\
&
=
\sum_{i,j=1}^\infty
|\alpha_i \sqrt{\rho_i} \alpha _j  \sqrt{\rho_j} |
\left|
\frac{1}{\sqrt{\rho_i}}
\frac{1}{\sqrt{\rho_j}}
(\bh_i , \cB^{(0)} \bh_j)_{L^{2,p}_{D}}
-
\frac{1}{\sqrt{\rho_i}}
\frac{1}{\sqrt{\rho_j}}
(\bh_i , \cB^{(1)} \bh_j)_{L^{2,p}_{D}}
\right|
\\
& \leq 
\sqrt{ \sum_{i,j=1}^\infty
\alpha_i^2 \rho_i \alpha _j^2 \rho_j  }
\sqrt{ \sum_{i=i,j}^\infty
\left(
\frac{1}{\sqrt{\rho_i}}
\frac{1}{\sqrt{\rho_j}}
(\bh_i , \cB^{(0)} \bh_j)_{L^{2,p}_{D}}
-
\frac{1}{\sqrt{\rho_i}}
\frac{1}{\sqrt{\rho_j}}
(\bh_j , \cB^{(1)} \bh_j)_{L^{2,p}_{D}}
\right)^2
}
\\
& \leq 
\sum_{k=1}^{\infty} \lambda_k^2.
\end{align*}

Hence, with $(\cB^{(a)})^{1/2}$ the unique operator square root of $\cB^{(a)}$ for $a=0,1$, there exists a finite constant $c$ such that for any $\bof \in L^{2,p}_{D}$, $((\cB^{(1)})^{1/2} \bof , (\cB^{(1)})^{1/2} \bof)_{L^{2,p}_{D}} \leq c ((\cB^{(0)})^{1/2} \bof , (\cB^{(0)})^{1/2} \bof)_{L^{2,p}_{D}}$. By a similar reasoning, there
exists a finite constant $c'$ such that for any $\bof \in L^{2,p}_{D}$, $((\cB^{(0)})^{1/2} \bof , (\cB^{(0)})^{1/2} \bof)_{L^{2,p}_{D}} \leq c' ((\cB^{(1)})^{1/2} \bof , (\cB^{(1)})^{1/2} \bof)_{L^{2,p}_{D}}$. Hence from Proposition B.1 in \cite{da2014stochastic}, the image spaces of $(\cB^{(0)})^{1/2}$ and $(\cB^{(1)})^{1/2}$ are the same.

Let $(\cB^{(0)})^{-1/2}$ be the pseudo inverse of $(\cB^{(0)})^{1/2}$ (see \cite{da2014stochastic}). Let also $\bpsi_k  = (\cB^{(0)})^{1/2}  \bh_k / \rho_k^{1/2}$. Then $(\bpsi_k)_{k \in \mathbb{N}}$ is an orthonormal basis of the image of $\cB^{(0)}$ in $L^{2,p}_{D}$. We obtain, recalling that the $\tilde{\bphi}_k $ for $k=1,\ldots,p$ are orthonormal in $\cW_{D} (\bF^{(0)})$, from \eqref{eq:bound:diff:scalar:product:phi}, \eqref{eq:from:WF2:to:B2} and \eqref{eq:from:WF1:to:B1}, 
\begin{align*}
\infty  
> \sum_{k=1}^{\infty} \lambda_k^2
& \geq 
\sum_{k,j=1}^{n}
\left(
( \tilde{\bphi}_k ,   \tilde{\bphi}_j )_{\cW_{D} (\bF^{(0)})}
-
( \tilde{\bphi}_k ,   \tilde{\bphi}_j )_{\cW_{D} (\bF^{(1)})}
\right)^2
\\
& = 
\sum_{k,j=1}^{n}
\left(
\frac{1}{\sqrt{\rho_k}}
\frac{1}{\sqrt{\rho_j}}
( \bphi_k ,   \bphi_j )_{\cW_{D} (\bF^{(0)})}
-
\frac{1}{\sqrt{\rho_k}}
\frac{1}{\sqrt{\rho_j}}
( \bphi_k ,   \bphi_j )_{\cW_{D} (\bF^{(1)})}
\right)^2
\\
& = 
\sum_{k,j=1}^{n}
\left(
\frac{1}{\sqrt{\rho_k}}
\frac{1}{\sqrt{\rho_j}}
( \bh_k ,  \cB^{(0)} \bh_j )_{L^{2,p}_{D}}
-
\frac{1}{\sqrt{\rho_k}}
\frac{1}{\sqrt{\rho_j}}
( \bh_k , \cB^{(1)}  \bh_j )_{L^{2,p}_{D}}
\right)^2
\\
& = \sum_{k,j=1}^{n}
\left( (   \bpsi_k , \bpsi_j   )_{L^{2,p}_{D}} 
-
 ( (\cB^{(0)})^{-1/2}  \bpsi_k , \cB^{(1)} (\cB^{(0)})^{-1/2} \bpsi_j )_{L^{2,p}_{D}} 
\right)^2
\\ 
& = 
\sum_{k,j=1}^{n}
\left( (   \bpsi_k , [I - (\cB^{(0)})^{-1/2} \cB^{(1)} (\cB^{(0)})^{-1/2}] \bpsi_j   )_{L^{2,p}_{D}}  
\right)^2.
\end{align*}

Hence, with the orthonormal basis $(\bpsi_k)_{k \in \IN}$ of the image of $\cB^{(0)}$ in $L^{2,p}_{D}$, we have
\[
\sum_{k,j=1}^{\infty}
\left( (   \bpsi_k , [I - (\cB^{(0)})^{-1/2} \cB^{(1)} (\cB^{(0)})^{-1/2}] \bpsi_j   )_{L^{2,p}_{D}}  
\right)^2
< \infty.
\]
Hence, $[I - (\cB^{(0)})^{-1/2} \cB^{(1)} (\cB^{(0)})^{-1/2}]$ is a Hilbert--Schmidt operator from the image of $\cB^{(0)}$ to $L^{2,p}_{D}$. Since we have seen also that the images of $(\cB^{(0)})^{1/2}$ and $(\cB^{(1)})^{1/2}$ are the same, from Theorem 2.25 in \cite{da2014stochastic} (see also Chapter 1 of \cite{maniglia2004gaussian}), $P(\bR^{(0)})\equiv P(\bR^{(1)})$ on the paths of $\bZ$. 
\end{proof}

\begin{proof}[{\bf Proof of Theorem \ref{theorem:sufficient:integral:normalized:difference:small}}] 

In this proof, it is convenient to consider two real-valued stationary $p$-variate Gaussian random fields $\{\bZ^{(0)}(\s) = (Z^{(0)}_{1}(\s) ,\ldots, Z^{(0)}_{p}(\s))^\top,\linebreak[1] \ss \in D \}$ and $\{\bZ^{(1)}(\s) = (Z^{(1)}_{1}(\s) , \ldots, Z^{(1)}_{p}(\s))^\top , \ss \in D \}$, where $\bZ^{(0)}$ and $\bZ^{(1)}$ have continuous sample paths, and where, for $a=0,1$, $\bZ^{(a)}$ has zero mean and matrix covariance function
 $\bR^{(a)} = [R^{(a)}_{ij}]_{i,j=1}^{p}$.

Let us first assume that for all $\blambda \in \IR^d$, $\bF^{(1)}(\blambda) \geq \bF^{(0)}(\blambda)$. Let, for $\blambda \in \IR^d$, with $c_1$ as in Condition \ref{cond:phi:zero},
\begin{align*}
\tilde{\bF}{}^{(0)}(\blambda)
& =
c_1 \gamma(\blambda)^2 \I_p ,
\\
\tilde{\bF}{}^{(1)}(\blambda)
& =
c_1 \gamma(\blambda)^2  \I_p
+ \bF^{(1)}(\blambda) - \bF^{(0)}(\blambda), \qquad \text{and}
\\
\tilde{\bF}{}^{(2)}(\blambda)
& =
\bF^{(0)}(\blambda) -
c_1 \gamma(\blambda)^2 \I_p.
\end{align*}
Let $\tilde{\bZ}{}^{(0)}$, $\tilde{\bZ}{}^{(1)}$ and $\tilde{\bZ}{}^{(2)}$ be three independent $p$-variate Gaussian processes with mean function zero and respective matrix spectral densities $\tilde{\bF}{}^{(0)}$, $\tilde{\bF}{}^{(1)}$ and $\tilde{\bF}{}^{(2)}$.
Then, in distribution, for $a=0,1$, $\bZ^{(a)} = \tilde{\bZ}{}^{(a)} + \tilde{\bZ}{}^{(2)} $. Hence, in order to prove the theorem, it is sufficient to show that the Gaussian measures given by $\tilde{\bZ}{}^{(0)}$ and $\tilde{\bZ}{}^{(1)}$ are equivalent. Let us do this.

Let $(\tilde{\bg}_k)_{k \in \IN}$ be an orthonormal basis in $\cW_{D}(\tilde{\bF}{}^{(0)})$. Let $n \in \IN$. Since $\cW_{D}$ is dense everywhere in $\cW_{D}(\tilde{\bF}{}^{(0)})$, from Lemma \ref{lem:from:tilde:g:to:g}, we can find $\bg_1,\ldots,\bg_n$ in $\cW_{D}$ for which $(\bg_k,\bg_l)_{\cW_{D}(\tilde{\bF}{}^{(0)})} = \delta_{k,l}$ for $k,l=1,\ldots,n$ and
\begin{equation} \label{eq:from:tilde:g:to:g}
\sum_{k=1}^{n}
\left[
\|\tilde{\bg}_k\|^2_{\cW_{D}(\tilde{\bF}{}^{(1)})}
-
\|\tilde{\bg}_k\|^2_{\cW_{D}(\tilde{\bF}{}^{(0)})}
\right]^2
\leq 
1+
\sum_{k=1}^{n}
\left[
\|\bg_k\|^2_{\cW_{D}(\tilde{\bF}{}^{(1)})}
-
\|\bg_k\|^2_{\cW_{D}(\tilde{\bF}{}^{(0)})}
\right]^2.
\end{equation}

We have, using Cauchy--Schwarz inequality,
\begin{align*}
\sum_{k=1}^{n}
\Bigl[
\|\bg_k & \|^2_{\cW_{D}(\tilde{\bF}{}^{(1)})}
-
\|\bg_k\|^2_{\cW_{D}(\tilde{\bF}{}^{(0)})}
\Bigr]^2
 =
\sum_{k=1}^{n}
\left[
\int_{\IR^d}
\bar{\bg}_k(\blambda)^\top
\left(  \tilde{\bF}{}^{(1)}(\blambda) - \tilde{\bF}{}^{(0)}(\blambda) \right)
\bg_k(\blambda)
d \blambda
\right]^2
\\
& \leq
\sum_{k=1}^{n}
\left[
\int_{\IR^d}
\| \bg_k(\blambda)\|
\gamma(\blambda)
\bnorml  \tilde{\bF}{}^{(1)}(\blambda) - \tilde{\bF}{}^{(0)}(\blambda) \bnormr
\frac{1}{\gamma(\blambda)}
\| \bg_k(\blambda)\|
d \blambda
\right]^2
%%%%%%%#######################
%%%%%%%NON-robust page break mea culpa
%\\
\end{align*}
\begin{align*}
& \leq 
\sum_{k=1}^{n}
\left(
\int_{\IR^d}
\| \bg_k(\blambda)\|^2
\gamma(\blambda)^2
d \blambda
\right)
%\\
%&
\left(
\int_{\IR^d}
\bnorml  \tilde{\bF}{}^{(1)}(\blambda) - \tilde{\bF}{}^{(0)}(\blambda) \bnormr^2
\frac{1}{\gamma(\blambda)^2}
\| \bg_k(\blambda)\|^2
d \blambda
\right)
\\
& \leq 
\frac{p2^dT^d}{c_1^3 (2 \pi)^d}
\int_{\IR^d}
\bnorml  \tilde{\bF}{}^{(1)}(\blambda) - \tilde{\bF}{}^{(0)}(\blambda) \bnormr^2
\frac{1}{\gamma(\blambda)^4}
d \blambda,
\end{align*}
from Lemma \ref{lemma:sum:gk^2}, Condition \ref{cond:phi:zero}, and since $\| \bg_k \|_{\cW_{D}(\tilde{\bF}{}^{(0)})} = 1$. The above display is finite and does not depend on $n$. We thus have
\[
\sum_{k=1}^{\infty}
\left[
\|\tilde{\bg}_k\|^2_{\cW_{D}(\tilde{\bF}{}^{(1)})}
-
\|\tilde{\bg}_k\|^2_{\cW_{D}(\tilde{\bF}{}^{(0)})}
\right]^2
< \infty.
\]

Let us define $V$ as a symmetric operator on $\cW_{D}( \tilde{\bF}{}^{(0)} )$ defined by $(V \bg , \bh)_{\cW_{D}(\tilde{\bF}{}^{(0)})} = \int_{\IR^d} \bar{\bg}(\blambda)^\top \tilde{\bF}{}^{(1)}(\blambda) \bh(\blambda) d \blambda  $ (its existence follows from Riesz representation theorem). 
Hence, $\sum_{k=1}^{\infty} [ ( \tilde{\bg}_k , (V-I) \tilde{\bg}_k)_{\cW_{D}( \tilde{\bF}{}^{(0)} )} ]^2< \infty$.
Recall that $\tilde{\bF}{}^{(1)} \geq \tilde{\bF}{}^{(0)}$ so $(V-I)^{1/2}$ is well-defined and compact, so from the spectral theorem, there exists a sequence of eigenfunctions  $(\bg_k)_{k \in \IN}$  of $(V - I)^{1/2}$ with the corresponding eigenvalues $( \sqrt{\alpha_k})_{k \in \IN}$. Furthermore, we have $\sum_{k=1}^{\infty}  (\sqrt{\alpha_k})^4 < \infty$. Hence, $(\bg_k)_{k \in \IN}$ is a sequence of eigenfunctions of $V-I$ with eigenvalues $(\alpha_k)_{k \in \IN}$ with $\sum_{k=1}^{\infty} \alpha_k^2 < \infty$ and so $V-I$ is Hilbert--Schmidt.

We have, for $r,q \in \IN$, with constants $c, c'$ and using the equivalence of matrix norms and that $(\bg_k)_{k \in \IN}$ is an orthonormal basis of $\cW_{D}(\tilde{\bF}{}^{(0)})$,
\begin{flalign*}
& \int_{\IR^d} 
\int_{\IR^d}
\Bnorml
\sum_{k=r}^{r+q}
\alpha_k \bg_k(\blambda)  \bar{\bg}_k( \bmu )^\top
\Bnormr^2
\| \tilde{\bF}{}^{(0)}(\blambda) \|\,
 \| \tilde{\bF}{}^{(0)}(\bmu) \| 
 d \blambda d \bmu
 & \\
 & \leq 
 c
 \int_{\IR^d} 
\int_{\IR^d}
\sum_{k=r}^{r+q}
\Bnorml
\tilde{\bF}{}^{(0)}(\blambda)^{1/2}
\left(
\alpha_k \bg_k(\blambda)  \bar{\bg}_k( \bmu )^\top
\right)
\tilde{\bF}{}^{(0)}(\bmu)^{1/2}
\Bnormr^2
 d \blambda d \bmu 
 &
 \\
 & \leq 
 c'
 \int_{\IR^d} 
\int_{\IR^d}
\trace \left(
\biggl[\;
\sum_{k=r}^{r+q}
\alpha_k
\tilde{\bF}{}^{(0)}(\blambda)^{1/2}
 \bg_k(\blambda)  \bar{\bg}_k( \bmu )^\top
\tilde{\bF}{}^{(0)}(\bmu)^{1/2}
\biggr]
\right.
\\
& \quad
\left.
\biggl[\;
\sum_{\ell=r}^{r+q}
\alpha_\ell
\tilde{\bF}{}^{(0)}(\bmu)^{1/2}
\bg_\ell(\bmu) \bar{\bg}_\ell( \blambda )^\top
\tilde{\bF}{}^{(0)}(\blambda)^{1/2}
\biggr]
\right)
 d \blambda d \bmu 
 &
 \\
 & =
 c' \sum_{k,\ell = r}^{r+q}
 \alpha_k \alpha_{\ell}
  \int_{\IR^d} 
\int_{\IR^d}
\trace 
\left(
\tilde{\bF}{}^{(0)}(\blambda)^{1/2}
 \bg_k(\blambda)  \bar{\bg}_k( \bmu )^\top
\tilde{\bF}{}^{(0)}(\bmu)^{1/2}
\tilde{\bF}{}^{(0)}(\bmu)^{1/2}
 \bg_\ell(\bmu) \bar{\bg}_\ell( \blambda )^\top
\tilde{\bF}{}^{(0)}(\blambda)^{1/2}
\right)
d \blambda d \bmu
\\
& =
 c' \sum_{k,\ell = r}^{r+q}
 \alpha_k \alpha_{\ell}
  \int_{\IR^d} 
\int_{\IR^d}
\bar{\bg}_\ell( \blambda )^\top
\tilde{\bF}{}^{(0)}(\blambda)
 \bg_k(\blambda)
   \bar{\bg}_k( \bmu )^\top
\tilde{\bF}{}^{(0)}(\bmu)
 \bg_\ell(\bmu) 
d \blambda d \bmu
\\
& =
\sum_{k=r}^{r+q} \alpha_k^2,
\end{flalign*}
since $(\bg_k)_{k \in \IN}$ is an orthonormal basis of $\cW_{D}(\tilde{\bF}{}^{(0)})$. Since $(\alpha_k)_{k \in \IN}$ is square summable, it follows that 
\[
\bB(\bmu , \blambda)^\top
=
\sum_{k=1}^{\infty}
\alpha_k
\bg_k(\blambda)
\bar{\bg}_k(\bmu)^\top
\]
is well-defined as a limit of Cauchy sequence and we have
\[
\int_{\IR^d} 
\int_{\IR^d}
\bnorml
\bB( \blambda , \bmu )
\bnormr^2
\| \tilde{\bF}{}^{(0)}(\blambda) \|\,
 \| \tilde{\bF}{}^{(0)}(\bmu) \| 
 d \blambda d \bmu < \infty.
\]
Let $\tilde{\bR}{}^{(j)}$ be the matrix covariance function of $\tilde{\bZ}_j$ for $j=0,1$. For $i=1,\ldots,p$, let
$\bpsi_{i,\r}(\blambda) = e^{ - \i \blambda^\top \r} \bfe_i$
for $\r \in \IR^d$. We have, for $i,j=1,\ldots,p$,
\begin{align*}
\bigl[
\tilde{\bR}{}^{(1)}( \t -\s )
&-
\tilde{\bR}{}^{(0)}(\t-\s)
\bigr]_{i,j}
=
\bfe_i^\top
\Bigl(
\int_{ \IR^d }
e^{ \i (\t-\s)^\top \blambda }
\tilde{\bF}{}^{(1)}( \blambda )
d \blambda 
\Bigr)
\bfe_j
-
\bfe_i^\top
\Bigl(
\int_{ \IR^d }
e^{ \i (\t-\s)^\top \blambda }
\tilde{\bF}{}^{(0)}( \blambda )
d \blambda 
\Bigr)
\bfe_j
\\
& = \int_{\IR^d}
\bar{\bpsi}_{i,\t}(\blambda)^\top
\tilde{\bF}{}^{(1)}(\blambda)
\bpsi_{j,\s}(\blambda)
d \blambda
-
\int_{\IR^d}
\bar{\bpsi}_{i,\t}(\blambda)^\top
\tilde{\bF}{}^{(0)}(\blambda)
\bpsi_{j,\s}(\blambda)
d \blambda
\\
& =
\left( 
(V-I) \bpsi_{i,\t}
,
\bpsi_{j,\s}
\right)_{\cW_{D}(\tilde{\bF}{}^{(0)})}
\\
& =
\sum_{k=1}^{\infty}
\alpha_k
\left( 
\bpsi_{i,\t}
,
\bg_k
\right)_{\cW_{D}(\tilde{\bF}{}^{(0)})}
\left( 
\bg_k
,
\bpsi_{j,\s}
\right)_{\cW_{D}(\tilde{\bF}{}^{(0)})}
\\
& =
\sum_{k=1}^{\infty}
\alpha_k
\int_{\IR^d}
\int_{\IR^d}
\bfe_i^\top e^{  \i \blambda^\top \t }
\tilde{\bF}{}^{(0)}(\blambda)
\bg_k(\blambda)
\bar{\bg}_k(\bmu)^\top 
\tilde{\bF}{}^{(0)}(\bmu) 
e^{ - \i \bmu^\top \s }
\bfe_j
d \blambda 
d \bmu
\\
& =
\bfe_i^\top
\Bigl(
\int_{\IR^d}
\int_{\IR^d}
 e^{   \i (\blambda^\top \t - \bmu^\top \s) }
\tilde{\bF}{}^{(0)}(\blambda)
\bB( \bmu , \blambda )^\top
\tilde{\bF}{}^{(0)}(\bmu) 
d \blambda 
d \bmu
\Bigr)
\bfe_j.
\end{align*}
Hence using $\tilde{\bR}^{(a)}(-\h) = \tilde{\bR}^{(a)}(\h)^\top$ for $a=0,1$ and $\tilde{\bF}{}^{(0)} = (\tilde{\bF}{}^{(0)})^\top$, we obtain
\[
\tilde{\bR}{}^{(1)}( \t -\s )
-
\tilde{\bR}{}^{(0)}(\t-\s)
=
\int_{\IR^d}
\int_{\IR^d}
 e^{  - \i \blambda^\top \t} 
e^{ \i \bmu^\top \s }
\tilde{\bF}{}^{(0)}(\blambda)
\bB( \blambda , \bmu )
\tilde{\bF}{}^{(0)}(\bmu) 
d \blambda 
d \bmu.
\]

Hence, from Theorem \ref{theorem:sufficient:representation:b}, the Gaussian measures given by $\tilde{\bZ}{}^{(0)}$ and $\tilde{\bZ}{}^{(1)}$ are equivalent, thus so are the Gaussian measures given by $\bZ^{(0)}$ and $\bZ^{(1)}$ as remarked previously.

Let us now drop the assumption that for all $\blambda \in \IR^d$, $\bF^{(1)}(\blambda) \geq \bF^{(0)}(\blambda)$.
Let $t^+$ be the positive part of $t \in \mathbb{R}$ and let, for $\blambda \in \IR^d$,
\begin{align*}
\tilde{\bF}{}^{(0)}(\blambda)
& =
c_1 \gamma(\blambda)^2 \I_p ,
\\
\tilde{\bF}{}^{(1)}(\blambda)
& =
c_1 \gamma(\blambda)^2  \I_p
+
\Bigl[
\sup_{\|\bv\|=1} \bigl(
\bar{\bv}^\top \bF^{(1)}(\blambda) \bv
-
\bar{\bv}^\top \bF^{(0)}(\blambda) \bv
\bigr)
\Bigr]^+
\I_p,
\\
\tilde{\bF}{}^{(2)}(\blambda)
& =
\bF^{(0)}(\blambda)
-
c_1 \gamma(\blambda)^2 \I_p \qquad \text{and}
\\
\bF^{(3)}(\blambda)
&
=
\tilde{\bF}{}^{(1)}(\blambda)
+
\tilde{\bF}{}^{(2)}(\blambda).
\end{align*}

Let $\tilde{\bZ}{}^{(0)}$, $\tilde{\bZ}{}^{(1)}$, $\tilde{\bZ}{}^{(2)}$ and $\bZ^{(3)}$ be four independent $p$-variate Gaussian processes with mean function zero and respective matrix spectral densities $\tilde{\bF}{}^{(0)}$, $\tilde{\bF}{}^{(1)}$, $\tilde{\bF}{}^{(2)}$ and $\bF^{(3)}$.

Then, in distribution, $\bZ^{(0)} = \tilde{\bZ}{}^{(0)} + \tilde{\bZ}{}^{(2)}$ and $ \bZ^{(3)} =  \tilde{\bZ}{}^{(1)} + \tilde{\bZ}{}^{(2)}$.
We have
\[
\int_{\IR^d}
\frac{1}{\gamma( \blambda )^4}
\|
\tilde{\bF}{}^{(0)}(\blambda)
-
\tilde{\bF}{}^{(1)}(\blambda)
\|^2
d \blambda
\leq 
\int_{\IR^d}
\frac{1}{\gamma( \blambda )^4}
\|
\bF^{(0)}(\blambda)
-
\bF^{(1)}(\blambda)
\|^2
d \blambda
< + \infty.
\]
Hence the Gaussian measures given by $\bZ^{(0)}$ and $\bZ^{(3)}$ are equivalent. 

We remark that for any $\blambda \in \IR^d$ and $\bv \in \IC^p$ with $\|\bv\| = 1$, we have
\begin{align*}
\bar{\bv}^\top &\bF^{(3)}(\blambda) \bv
- 
\bar{\bv}^\top \bF^{(1)}(\blambda) \bv
\\
&
~ ~
=
\bar{\bv}^\top \bF^{(0)}(\blambda) \bv
- 
\bar{\bv}^\top \bF^{(1)}(\blambda) \bv
+
\Bigl[
\sup_{\|\bv\|=1} \bigl(
\bar{\bv}^\top \bF^{(1)}(\blambda) \bv
-
\bar{\bv}^\top \bF^{(0)}(\blambda) \bv
\bigr)
\Bigr]^+
%&
%\\
%&
~ 
\geq 0.
\end{align*}
Hence, applying the previous step of the proof, the measures given by $\bZ^{(1)}$ and $\bZ^{(3)}$ are equivalent since
\begin{flalign*}
&
\int_{\IR^d}
\frac{1}{\gamma(\blambda)^4}
\bnorml
\bF^{(3)}(\blambda) - \bF^{(1)}(\blambda)
\bnormr^2
d \blambda 
\leq 
4
\int_{\IR^d}
\frac{1}{\gamma(\blambda)^4}
\bnorml
\bF^{(0)}(\blambda) - \bF^{(1)}(\blambda)
\bnormr^2
d \blambda 
 < \infty.
 &
\end{flalign*}

Hence the Gaussian measures given by $\bZ^{(0)}$ and $\bZ^{(1)}$ are equivalent.
\end{proof}

The next lemma is immediate.

\begin{lemma} \label{lemma:fdeux:smaller:constant:Fun}
Assume that Condition \ref{cond:phi:zero} holds.
With $ c_3 = c_2 / c_1 $, for $\blambda \in \IR^d$,
\[
\bF^{(1)}(\blambda) \leq c_3 \bF^{(0)}(\blambda).
\]
\end{lemma}

\begin{lemma} \label{lemma:sum:gk^2}
In the context of the proof of Theorem \ref{theorem:sufficient:integral:normalized:difference:small}, for $n \in \IN$, let
$\bg_1,\ldots,\bg_n$ in $\cW_{D}$ for which $(\bg_k,\bg_l)_{\cW_{D}(\tilde{\bF}{}^{(0)})} = \delta_{k,l}$ for $k,l=1,\ldots,n$. Then for any $\blambda \in \IR^d$, with a finite constant $T$,
\[
\sum_{ k=1 }^n
\| \bg_k(\blambda) \|^2 
\leq 
\frac{p2^dT^d}{ \gamma^2( \blambda ) c_1^2 (2 \pi)^d }.
\]
\end{lemma}

\begin{proof}[{\bf Proof of Lemma \ref{lemma:sum:gk^2}}]

For $\blambda \in \IR^d$ and $k = 1,\ldots,n$, let $\bh_k(\blambda) = (h_{k,1}(\blambda) ,\ldots, h_{k,p}(\blambda))^\top = \gamma(\blambda) \bg_k( \blambda )$.
By convolution, with a finite constant $T$, there exist $p$ square summable functions $\psi_{k,1} , \ldots , \psi_{k,p}$ from $[-T,T]^d$ to $\IC$ such that for $i=1,\ldots,p$
\[
h_{k,i}(\blambda) = \int_{[-T,T]^d} e^{- \i \blambda^\top \t} \psi_{k,i}(\t) d \t.
\]
We have, for $k,\ell  = 1,\ldots,n$
\[
c_1
\int _{\IR^d}
\bar{\bg}_k(\blambda)^\top 
\gamma(\blambda)^2
\bg_\ell(\blambda)
d \blambda
= \delta_{k,\ell}
\]
and thus
\[
c_1
\int_{\IR^d}
\bar{\bh}_{k}(\blambda)^\top
 \bh_{\ell}(\blambda)
d \blambda 
=
\delta_{k,\ell}.
\]
By Plancherel's theorem we obtain
\[
c_1
\int_{[-T,T]^d}
\bar{\psi}_{k,1}(\t) \psi_{\ell,1}(\t)
d \t
+
\dots
+
c_1
\int_{[-T,T]^d}
\bar{\psi}_{k,p}(\t) \psi_{\ell,p}(\t)
d \t
=
\frac{1}{(2 \pi)^d}
\delta_{k,\ell}.
\]

Hence, $ c_1 (2 \pi)^{d/2} ((\psi_{k,1} , \ldots , \psi_{k,p})^\top)_{k =1,\ldots,n} = c_1 (2 \pi)^{d/2} (\bpsi_{k})_{k =1,\ldots,n}$ is an orthonormal system in $L^{2,p}_{[-T,T]}$.

For $i=1,\ldots,p$, let $\bphi_{\blambda,i}(\t) = ( 0 ,\ldots,0 ,  e^{\i \blambda^\top \t}  , 0 , \ldots ,0)^\top$ for $\blambda, \t \in \IR^d$, where the non-zero element is at position $i$. From Bessel's inequality we obtain
\begin{align*}
\sum_{k=1}^{n}
| h_{k,i}(\blambda) |^2
& =
\sum_{k=1}^{n}
\left|
\int_{[-T,T]^d}
 \psi_{k,i}(\t)
  e^{ - \i \blambda^\top \t}
  d \t
\right|^2
 =
\sum_{k=1}^{n}
\left|
\int_{[-T,T]^d}
\bar{\bphi}_{\blambda,i}(\t)^\top \bpsi_k(\t) d \t
\right|^2
\\
& \leq 
\frac{1}{c_1^2 (2\pi)^d}
\int_{[-T,T]^d}
\| \bphi_{\blambda,i}(\t) \|^2 d \t
 \leq 
\frac{1}{c_1^2 (2\pi)^d}
(2T)^d.
\end{align*}
Hence
\[
\sum_{k=1}^n
\gamma(\blambda)^2
\|\bg_k(\blambda)\|^2
=
\sum_{i=1}^p
\sum_{k=1}^{n}
|h_{k,i}(\blambda)|^2
\leq
\frac{
p(2T)^d
}
{
c_1^2 (2 \pi)^d
}.
\]
\end{proof}

\begin{lemma} \label{lem:from:tilde:g:to:g}
We can find $\bg_1,\ldots,\bg_n$ in $\cW_{D}$ as described in \eqref{eq:from:tilde:g:to:g}. 
\end{lemma}

\begin{proof}[{\bf Proof of Lemma \ref{lem:from:tilde:g:to:g}}]

By density, for $\epsilon >0$, we can find $\hat{\bg}_1,\ldots,\hat{\bg}_n$ in $\cW_{D}$ such that for $k=1,\ldots,n$, $\| \hat{\bg}_k - \tilde{\bg}_k \|_{\cW_{D}(\tilde{\bF}{}^{(0)})} \leq \epsilon / n^2$. Let $\bM = [ (  \hat{\bg}_k , \hat{\bg}_\ell)_{\cW_{D}(\tilde{\bF}{}^{(0)})} ]_{k,\ell = 1,\ldots,n}$. We have, for $k,\ell=1,\ldots,n$,
\begin{align} \label{eq:diff:innder:product:un}
\left|
(  \hat{\bg}_k , \hat{\bg}_\ell)_{\cW_{D}(\tilde{\bF}{}^{(0)})}
-
(  \tilde{\bg}_k , \tilde{\bg}_\ell)_{\cW_{D}(\tilde{\bF}{}^{(0)})}
\right|
&
\leq
\left|
(  \hat{\bg}_k - \tilde{\bg}_k , \hat{\bg}_\ell)_{\cW_{D}(\tilde{\bF}{}^{(0)})}
\right|
+
\left|
(  \tilde{\bg}_k , \hat{\bg}_\ell - \tilde{\bg}_{\ell})_{\cW_{D}(\tilde{\bF}{}^{(0)})}
\right|  \nonumber
\\
&
\leq 
\frac{3 \epsilon}{n^2},
\end{align}
from Cauchy--Schwarz.
Also, from Lemma \ref{lemma:fdeux:smaller:constant:Fun}, 
\begin{align} \label{eq:diff:innder:product:deux}
\left|
(  \hat{\bg}_k , \hat{\bg}_\ell)_{\cW_{D}(\tilde{\bF}{}^{(1)})}
-
(  \tilde{\bg}_k , \tilde{\bg}_\ell)_{\cW_{D}(\tilde{\bF}{}^{(1)})}
\right|
&\leq
\left|
(  \hat{\bg}_k - \tilde{\bg}_k , \hat{\bg}_\ell)_{\cW_{D}(\tilde{\bF}{}^{(1)})}
\right|
+
\left|
(  \tilde{\bg}_k , \hat{\bg}_\ell - \tilde{\bg}_{\ell})_{\cW_{D}(\tilde{\bF}{}^{(1)})}
\right|
\nonumber\\
&
\leq 
\frac{c_3 3 \epsilon}{n^2}.
\end{align}
From Gershogrin's circle theorem, we thus have $\|\bM - \I_n\| \leq 3 \epsilon / n$. For $k=1,\ldots,n$, let $\bg_k = \sum_{\ell=1}^n ( \bM^{-1/2} )_{k,\ell} \hat{\bg}_{\ell}$. Then, $\bg_1,\ldots,\bg_n$ satisfy $(\bg_k,\bg_l)_{\cW_{D}(\tilde{\bF}{}^{(0)})} = \delta_{k,l}$ for $k,l=1,\ldots,n$. Furthermore, there exists a constant $c_n < \infty$, not depending on $\epsilon$, such that $\| \bg_k - \tilde{\bg}_k \|_{\cW_{D}(\tilde{\bF}{}^{(0)})} \leq c_n \epsilon / n^2$.
 
We then have, with a constant $c'_n$ not depending on $\epsilon$,
 \begin{align*}
 \left|
\sum_{k=1}^{n} \right.
\Bigl[ 
\|\tilde{\bg}_k&\|^2_{\cW_{D}(\tilde{\bF}{}^{(1)})}
-
\|\tilde{\bg}_k\|^2_{\cW_{D}(\tilde{\bF}{}^{(0)})}
\Bigr]^2
-
\left.\sum_{k=1}^{n}
\left[
\|\bg_k\|^2_{\cW_{D}(\tilde{\bF}{}^{(1)})}
-
\|\bg_k\|^2_{\cW_{D}(\tilde{\bF}{}^{(0)})}
\right]^2
\right|
&
\\
&
=
\left|
\sum_{k=1}^{n}
\left(
\|\tilde{\bg}_k\|^2_{\cW_{D}(\tilde{\bF}{}^{(1)})}
-
\|\bg_k\|^2_{\cW_{D}(\tilde{\bF}{}^{(1)})}
-
\|\tilde{\bg}_k\|^2_{\cW_{D}(\tilde{\bF}{}^{(0)})}
+
\|\bg_k\|^2_{\cW_{D}(\tilde{\bF}{}^{(0)})}
\right)
\right.
\\
& ~ ~
\left.
\left(
\|\tilde{\bg}_k\|^2_{\cW_{D}(\tilde{\bF}{}^{(1)})}
+
\|\bg_k\|^2_{\cW_{D}(\tilde{\bF}{}^{(1)})}
-
\|\tilde{\bg}_k\|^2_{\cW_{D}(\tilde{\bF}{}^{(0)})}
-
\|\bg_k\|^2_{\cW_{D}(\tilde{\bF}{}^{(0)})}
\right)
\right|
\\
& \leq 
c'_n
\frac{\epsilon}{n^2},
\end{align*}
proceeding as for \eqref{eq:diff:innder:product:un} and \eqref{eq:diff:innder:product:deux}. This concludes the proof since $\epsilon$ can be chosen arbitrarily small.
\end{proof}

\section{Proofs for Section \ref{SECTION:APPLICATION:MATERN:WENDLAND}}

Following notation in Section \ref{sec2}, we denote with $P({\cal BM_{\btheta}})$ a zero mean Gaussian measure induced by a bivariate Mat{\'e}rn covariance function
with associated spectral density $\widehat{{\cal BM}}_{\btheta}$,
 and
with $P( {\cal BW}_{\blambda})$
 a zero mean Gaussian measure induced by a Generalized Wendland  covariance function with associated spectral density  $\widehat{{\cal BW}}_{\blambda}$.

\subsection{Mathematical Background for the Proofs}

Fourier transforms of isotropic covariance functions, for a given $d$, have a simple expression, as reported in \cite{Yaglom:1987} or \cite{Stein:1999}. For a continuous function $\phi:[0,\infty) \to \R$ such  that $\phi(\|\cdot\|)$ is positive definite in $\R^d$, we define its isotropic spectral density as
\begin{equation} \label{FT}
 \widehat{\phi}(z)= \frac{z^{1-d/2}}{(2 \pi)^d} \int_{0}^{\infty} u^{d/2} J_{d/2-1}(uz)  \phi(u) {\rm d} u, \qquad z \ge 0,
\end{equation}
where $J_{d/2-1}$ is the Bessel function of the first kind.
Throughout the paper, we use the notations: $\widehat{{\cal M}}_{\nu,\alpha}$, and 
$\widehat{{\cal W}}_{\mu,\kappa,\beta}$
 for the isotropic spectral densities
of ${{\cal M}}_{\nu,\alpha}$ and ${\cal W}_{\mu,\kappa,\beta} $, respectively.

A well-known result about the spectral density of the  Mat{\'e}rn model is the following:
\begin{equation} \label{eq:FT:matern}
\widehat{{\cal M}}_{\nu,\alpha}(z)= \frac{\Gamma(\nu+d/2)}{\pi^{d/2} \Gamma(\nu)}
\frac{ \alpha^d}{(1+\alpha^2z^2)^{\nu+d/2}}
, \qquad z \ge 0.
\end{equation}

The spectral density of the Generalized Wendland  function is given by \cite{bevilacqua2019estimation}
$$\widehat{{\cal W}}_{\mu,\kappa,\beta}(z)=L(d,\mu,\kappa)\beta^{d}\mathstrut_1 F_2\Big(\zeta;\zeta+\frac{\mu}{2},\zeta+\frac{\mu}{2}+\frac{1}{2};-\frac{(z\beta)^{2}}4\Big), \quad  z > 0;$$\\
where $\zeta = (d+1)/2 + \kappa$, where 
\[
L(d,\mu,\kappa)=
\begin{cases}
\frac{2^{-d}\pi^{-d/2}\Gamma(\mu+1)\Gamma(2\kappa+d)}{\Gamma(\kappa+d/2)\Gamma(\mu+d+1+2\kappa)} \frac{\Gamma(\kappa)}{B(2\kappa,\mu+1)}
& \quad \text{if} \quad \kappa>0
\\
\frac{2^{-d+1}\pi^{-d/2}\Gamma(\mu+1)\Gamma(d)}{\Gamma(d/2)\Gamma(\mu+d+1)}
& \quad \text{if} \quad \kappa=0
\end{cases}
\]
and where $\mathstrut_1 F_2(a;b,c;z)$ is defined in \eqref{eq:oneFtwo}.

Then, the spectral densities associated with the multivariate Mat\'ern and Generalized Wendland models (\ref{eq:piip5}) and (\ref{eq:piip6}) are
\begin{equation*}%\label{eq:piip2}
\widehat{{\cal BM}}_{\btheta}(z)= \left[ \rho_{ij}\sigma_{ii}\sigma_{jj} \widehat{{\cal M}}_{\nu,\alpha_{ij}}(z) \right]_{i,j=1}^2 , \quad  \rho_{11}=\rho_{22} = 1, \qquad \alpha_{12} =  \alpha_{21}
\end{equation*}
and
\begin{equation*}%\label{eq:piip3}
\widehat{{\cal BW}}_{\blambda}(z)= \left[ \rho_{ij}\sigma_{ii}\sigma_{jj} \widehat{{\cal W}}_{\mu,\kappa,\beta_{ij}}(z) \right]_{i,j=1}^2 , \quad  \rho_{11}=\rho_{22} = 1, \qquad \beta_{12} =  \beta_{21}.
\end{equation*}

\subsection{Proofs of Theorems \ref{theorem:matern:vs:matern}, \ref{theorem:wendland:vs:wendland} and \ref{theorem:matern:vs:wendland}}
\label{app:cond}

Before proving Theorems \ref{theorem:matern:vs:matern}, \ref{theorem:wendland:vs:wendland} and \ref{theorem:matern:vs:wendland}, we state and prove several preliminary lemmas.

\begin{lemma} \label{lem:finite:integral}
For any $x>0$, $y >0$ and $r >0$, 
\[
\int_{\IR^d}
(1 + \|\blambda\|)^{4r}
\left(
\frac{1}{(x + \|\blambda\|^2)^r}
-
\frac{1}{(y + \|\blambda\|^2)^r}
\right)^2 d\blambda
< \infty.
\]
\end{lemma}
\begin{proof}[{\bf Proof of Lemma \ref{lem:finite:integral}}]
The integral is equal to
\begin{align*}
\int_{\IR^d} &
\frac{(1 + \|\blambda\|)^{4r}}{(x + \|\blambda\|^2)^{2r} (y + \|\blambda\|^2)^{2r}}
\left(
(y+\|\blambda\|^2)^{r}
-
(x+\|\blambda\|^2)^{r}
\right)^2
d\blambda 
\\
& \leq 
\int_{\IR^d}
\frac{(1 + \|\blambda\|)^{4r}}{(x + \|\blambda\|^2)^{2r} (y + \|\blambda\|^2)^{2r}}
\|\blambda\|^{4r}
\left(
\left( 1 + \frac{y}{\|\blambda\|^2} \right)^r
-
\left( 1 + \frac{x}{\|\blambda\|^2} \right)^r
\right)^2 
d\blambda
\\
&
\leq 
c
\int_{\IR^d}
\left(
\left( 1 + \frac{y}{\|\blambda\|^2} \right)^r
-
\left( 1 + \frac{x}{\|\blambda\|^2} \right)^r
\right)^2 
d\blambda 
\end{align*}
with a constant $0 < c <\infty$. The function in the above integral has order of magnitude $1/ \| \blambda \|^4 $ as $\|\blambda\| \to \infty$ so this function is integrable on $\IR^d$ for $d=1,2,3$.
\end{proof}

\begin{lemma} \label{lemma:matern:assympt:invertible}
Consider $\nu>0$, $\alpha_{11}>0$, $\alpha_{22}>0$, $\alpha_{12} >0$, $\sigma_{11} >0$, $\sigma_{22} >0$ and $\rho_{12}   \in (-1,1)$. Let $\btheta=(\sigma_{11},\sigma_{22},\rho_{12},\nu,\alpha_{11},\alpha_{22},\alpha_{12})^{\top}$.
If 
\begin{equation}
\rho_{12}^2
<
\frac{ \alpha_{12}^{4 \nu} }{ \alpha_{11}^{2 \nu} \alpha_{22}^{2 \nu}  }
\inf_{t \geq 0}
\frac{
( \alpha_{12}^{-2} +t^2 )^{2 \nu +d}
}{
( \alpha_{11}^{-2} +t^2 )^{\nu +d/2}
( \alpha_{22}^{-2} +t^2 )^{\nu +d/2}
}
\end{equation}
then, there exist constants $0 < c < c' < \infty$ such that for all $\blambda \in \IR^d$,
\[
c (1 + \|\blambda\|)^{-2\nu - d} \I_2
\leq 
\widehat{{\cal BM}}_{\btheta}(\|\blambda\|)
\leq 
c' (1 + \|\blambda\|)^{-2\nu - d} \I_2.
\]
\end{lemma}

\begin{proof}[{\bf Proof of Lemma \ref{lemma:matern:assympt:invertible}}]
It suffices to show that there exist constants $0 < c < c' < \infty$ such that
\[
c  \I_2
\leq 
(1 + \|\blambda\|)^{2\nu + d}
\widehat{{\cal BM}}_{\btheta}(\|\blambda\|)
\leq 
c'
\I_2.
\]
Because of \eqref{eq:FT:matern}, the elements of $(1 + \|\blambda\|)^{2\nu + d}
\widehat{{\cal BM}}_{\btheta}(\|\blambda\|)$ are bounded on $\IR^d$. It is thus sufficient to show that the determinant of $(1 + \|\blambda\|)^{2\nu + d}
\widehat{{\cal BM}}_{\btheta}(\|\blambda\|)$ is bounded away from $0$ on $\IR^d$. This determinant is equal to
\begin{align*}
&
(1 + \|\blambda\|)^{4\nu + 2d}
\left(
\frac{
\Gamma( \nu + d/2 )
}{
\Gamma(\nu) \pi^{d/2}
}
\right)^2
\left(
\frac{\sigma_{11}^2 \sigma_{22}^2 }{\alpha_{11}^{2 \nu} \alpha_{22}^{2 \nu}}
\frac{1}{(\alpha_{11}^{-2} + \|\blambda\|^2)^{\nu + d/2}}
\frac{1}{(\alpha_{22}^{-2} + \|\blambda\|^2)^{\nu + d/2}}
\right.
\\
& \quad \left.
-
\frac{(\sigma_{11} \sigma_{22} \rho_{12})^2}{\alpha_{12}^{4 \nu}}
\frac{1}{(\alpha_{12}^{-2} + \|\blambda\|^2)^{2\nu + d}}
\right)
\\
& =
\frac{
(1 + \|\blambda\|)^{4\nu + 2d}
}
{
\alpha_{12}^{4 \nu}
(\alpha_{12}^{-2} + \|\blambda\|^2)^{2\nu + d}
}
\left(
\frac{
\Gamma( \nu + d/2 )
}{
\Gamma(\nu) \pi^{d/2}
}
\right)^2
\sigma_{11}^2 \sigma_{22}^2
\\
& \quad
\left(
\frac{ \alpha_{12}^{4 \nu} }{ \alpha_{11}^{2 \nu} \alpha_{22}^{2 \nu} }
\frac{
(\alpha_{12}^{-2} + \|\blambda\|^2)^{2\nu + d}
}
{
(\alpha_{11}^{-2} + \|\blambda\|^2)^{\nu + d/2}
(\alpha_{22}^{-2} + \|\blambda\|^2)^{\nu + d/2}
}
-
\rho_{12}^2
\right),
\end{align*}
which is bounded away from $0$ by assumption.
\end{proof}

\begin{lemma} \label{lemma:wendland:assympt:invertible}
Consider $\kappa \ge 0$, $\mu > d+1/2+\kappa$, $\beta_{11}>0$, $\beta_{22}>0$, $\beta_{12} >0$, $\sigma_{11} >0$, $\sigma_{22} >0$ and $\rho_{12}   \in (-1,1)$.
 Let $\blambda=(\sigma_{11},\sigma_{22},\rho_{12},\mu,\kappa,\beta_{11},\beta_{22},\beta_{12})^{\top}$.
If 
\begin{equation} \label{eq:cond:inf:wendland2}
\rho_{12}^2
<
\inf_{z \geq 0}
\frac{  
 \widehat{{\cal W}}_{\mu,\kappa,\beta_{11}}(z) 
 \widehat{{\cal W}}_{\mu,\kappa,\beta_{22}}(z)  }{
\bigl( \widehat{{\cal W}}_{\mu,\kappa,\beta_{12}}(z) \bigr){}^2
}
\end{equation}
then, there exist constants $0 < c < c' < \infty$ such that for all $\blambda \in \IR^d$,
\[
c (1 + \|\blambda\|)^{-(d+1+2\kappa)} \I_2
\leq 
\widehat{{\cal BW}}_{\blambda}(\| \blambda \|)
\leq 
c' (1 + \|\blambda\|)^{-( d+1+2\kappa )} \I_2.
\]
\end{lemma}
\begin{proof}[{\bf Proof of Lemma \ref{lemma:wendland:assympt:invertible}}]
We remark that, for all $\beta>0$, there exist constants $0 < \tilde{c} < \tilde{c}' < \infty$ such that for all $\blambda \in \IR^d$,
\[
\tilde{c}
(1+\|\blambda\|)^{-(d+1+2\kappa)}
\leq 
\widehat{{\cal W}}_{\mu,\kappa,\beta}(\|\blambda\|)
\leq 
\tilde{c}'
(1+\|\blambda\|)^{-(d+1+2\kappa)}
\]
from Theorem 1 in \cite{bevilacqua2019estimation}.
The rest of the proof is similar to that of Lemma \ref{lemma:matern:assympt:invertible}. We remark that the right-hand side of \eqref{eq:cond:inf:wendland2} is equal to the right-hand side of \eqref{eq:cond:inf:wendland:longer}.
\end{proof}

\begin{lemma} \label{lemma:exist:phi_zero:matern}
Let $\nu >0$.
There exists a function $\gamma : \IR^d \to \IR$ such that $(\gamma,\ldots,\gamma)^\top \in \cW_{[-b,b]^d} $ for some fixed $0 < b < \infty $ and there exist two constants $0 < c < c' < \infty$ such that for all $\blambda \in \IR^d$
\[
c \gamma^2(\blambda)
\leq (1 + \|\blambda\|)^{- 2\nu -d}
\leq 
c' \gamma^2(\blambda).
\]
\end{lemma}

\begin{proof}[{\bf Proof of Lemma \ref{lemma:exist:phi_zero:matern}}]
In \cite{zastavnyi2006some}, it is shown that there exists a family of functions from $\mathbb{R}^d$ to $\mathbb{R}$ of the form $\phi_{\delta,\mu,\bar{\nu},\alpha}$, with compact support, for $\delta,\mu,\bar{\nu} >0$ and $\alpha \in \IR$, such that
by combining Theorem 3, 6. (iii) and theorem 6, 1., (i) in \cite{zastavnyi2006some}, the Fourier transform $\hat{\phi}_{\delta,\mu,\bar{\nu},\alpha}$ of $\phi_{\delta,\mu,\bar{\nu},\alpha}$ satisfies, for two positive finite constants $\tilde{c},\tilde{c}'$, for all $\blambda \in \IR^d$,
\begin{equation} \label{eq:rate:FT:compact}
\tilde{c} \hat{\phi}_{\delta,\mu,\bar{\nu},\alpha}(\blambda)
\leq (1 + \|\blambda\|)^{- (d-1+\alpha+1)}
\leq 
\tilde{c}' \hat{\phi}_{\delta,\mu,\bar{\nu},\alpha}(\blambda)
\end{equation}
provided
\begin{align*}
& \delta >0,
\quad \quad
\mu >0,
\quad \quad
\bar{\nu} >0,
\quad \quad
\alpha + d >0,
\quad \quad
d-1+\alpha+1 >0,
 \quad \quad
\delta>0,
\quad \quad
\mu>0,
\\
& \frac{d-1}{2} + \bar{\nu} + \frac{1}{2} >0,
\quad \quad
d-1+\alpha < 2 \left(  \frac{d-1}{2} + \bar{\nu}  \right),
\quad \quad
\mu + \frac{d-1}{2} + \bar{\nu} 
> d-1 + \alpha +1,
\\
&
\delta=2,
\quad \quad
\mu >\frac{1}{2},
\quad \quad
\frac{d-1}{2} + \bar{\nu} > - \frac{1}{2},
\quad \quad
d - 1 + \alpha = \frac{d-1}{2} + \bar{\nu} - \frac{1}{2}.
\end{align*}
These conditions hold if
\begin{align*}
& \bar{\nu} >0,
\quad \quad
d+\alpha >0,
\quad \quad
\delta=2,
\quad \quad
\mu> \frac{1}{2},
\quad \quad
d + \bar{\nu} >0,
\quad \quad
\alpha < 2 \bar{\nu},
\\
& \mu  -\frac{d}{2}  - \frac{1}{2} + \bar{\nu}
> \alpha,
\quad \quad
\bar{\nu} = \frac{d}{2} + \alpha.
\end{align*}
These conditions hold if
\begin{align*}
&
\frac{d}{2} + \alpha >0,
\quad \quad
d + \alpha >0,
\quad \quad
\delta=2,
\quad \quad
\mu> \frac{1}{2},
\quad \quad
\frac{3d}{2} + \alpha > 0 ,
\quad \quad
\alpha < d + 2 \alpha,
\\
&
\mu > \frac{1}{2},
\quad \quad
\bar{\nu} = \frac{d}{2} + \alpha.
\end{align*}
These conditions hold if
\begin{align*}
& \alpha > - \frac{d}{2},
\quad \quad
\alpha > - d,
\quad \quad
\delta=2,
\quad \quad
\mu> \frac{1}{2},
\quad \quad
\alpha > - \frac{3d}{2},
\quad \quad
\alpha > -d,
\quad \quad
\bar{\nu} = \frac{d}{2} + \alpha.
\end{align*}
Hence, we can select $\alpha = \nu - d/2$, in which case in \eqref{eq:rate:FT:compact}, we have $d-1+\alpha+1 = d/2 + \nu$ and thus we can let $\gamma = \hat{\phi}_{\delta,\mu,\bar{\nu},\alpha}$ to conclude the proof.
\end{proof}

\begin{proof}[{\bf Proof of Theorem \ref{theorem:matern:vs:matern}}]
From Lemmas \ref{lemma:matern:assympt:invertible} and \ref{lemma:exist:phi_zero:matern}, Condition \ref{cond:phi:zero} holds.
Let $\gamma^2$ be as in Lemma~\ref{lemma:exist:phi_zero:matern}, satisfying Condition \ref{cond:phi:zero}.
We have, for $i,j=1,2$, with a constant $0 < c < \infty$
\begin{align*}
\int_{\IR^d} &
\frac{1}{\gamma(\blambda)^4}
\left(
\rho_{ij}^{(0)}\sigma_{ii}^{(0)} \sigma_{jj}^{(0)} \widehat{{\cal M}}_{\nu,\alpha_{ij}^{(0)}}(\|\blambda\|)
-
\rho_{ij}^{(1)}\sigma_{ii}^{(1)} \sigma_{jj}^{(1)} \widehat{{\cal M}}_{\nu,\alpha_{ij}^{(1)}}(\|\blambda\|)
\right)^2
d \blambda
\\
&
\leq 
c
\int_{\IR^d}
(1+\| \blambda \|)^{4 \nu +2d}
\left\{
\frac{\sigma_{ii}^{(0)}
\sigma_{jj}^{(0)}
\rho_{ij}^{(0)}
}
{
(\alpha_{ij}^{(0)})^{2 \nu}
}
\frac{ \Gamma(\nu+d/2) }{ \Gamma(\nu) \pi^{d/2} }
\frac{1}{ ( (\alpha_{ij}^{(0)})^{-2} + \|\blambda\|^2  )^{\nu + d/2} }
\right.
\\
& \quad
\left.
-
\frac{\sigma_{ii}^{(1)}
\sigma_{jj}^{(1)}
\rho_{ij}^{(1)}
}
{
(\alpha_{ij}^{(1)})^{2 \nu}
}
\frac{ \Gamma(\nu+d/2) }{ \Gamma(\nu) \pi^{d/2} }
\frac{1}{ ( (\alpha_{ij}^{(1)})^{-2} + \|\blambda\|^2  )^{\nu + d/2} }
\right\}^2
d \blambda 
\\
& = 
c
\left(
\frac{\sigma_{ii}^{(0)}
\sigma_{jj}^{(0)}
\rho_{ij}^{(0)}
}
{
(\alpha_{ij}^{(0)})^{2 \nu}
}
\frac{ \Gamma(\nu+d/2) }{ \Gamma(\nu) \pi^{d/2} }
\right)^2
\int_{\IR^d}
(1+\| \blambda \|)^{4 \nu +2d}
\\
&  \quad
\left[
\frac{1}{ ( (\alpha_{ij}^{(0)})^{-2}  + \|\blambda\|^2 )^{\nu + d/2} }
-
\frac{1}{ ( (\alpha_{ij}^{(1)})^{-2}  + \|\blambda\|^2   )^{\nu + d/2} }
\right]^2
d\blambda
 % \\
 % &
< \infty
\end{align*}
from Lemma \ref{lem:finite:integral}. This concludes the proof, from Theorem \ref{theorem:sufficient:integral:normalized:difference:small}.
\end{proof}

 \begin{proof}[{\bf Proof of Theorem \ref{theorem:wendland:vs:wendland}}]
From Lemmas \ref{lemma:wendland:assympt:invertible} and \ref{lemma:exist:phi_zero:matern}, Condition \ref{cond:phi:zero} holds.
Let $\gamma^2$ be as in Lemma~\ref{lemma:exist:phi_zero:matern}, satisfying Condition~\ref{cond:phi:zero}.
From Theorem~\ref{theorem:sufficient:integral:normalized:difference:small}, in order to prove Theorem~\ref{theorem:wendland:vs:wendland}, it is sufficient to show that, for $i,j=1,2$,
\[
\int_{\IR^d}
\frac{
1
}{
\gamma(\blambda)^4
}
\left(
\rho_{ij}^{(0)}\sigma_{ii}^{(0)} \sigma_{jj}^{(0)} \widehat{{\cal W}}_{\mu,\kappa,\beta_{ij}^{(0)}}(\|\blambda\|)
-
\rho_{ij}^{(1)}\sigma_{ii}^{(1)} \sigma_{jj}^{(1)} \widehat{{\cal W}}_{\mu,\kappa,\beta_{ij}^{(1)}}(\|\blambda\|)
\right)^2
d \blambda
< \infty.
\]
This is equivalent to
\[
\int_{\IR^d}
\frac{
\left(
\rho_{ij}^{(0)}\sigma_{ii}^{(0)} \sigma_{jj}^{(0)} \widehat{{\cal W}}_{\mu,\kappa,\beta_{ij}^{(0)}}(\|\blambda\|)
-
\rho_{ij}^{(1)}\sigma_{ii}^{(1)} \sigma_{jj}^{(1)} \widehat{{\cal W}}_{\mu,\kappa,\beta_{ij}^{(1)}}(\|\blambda\|)
\right)^2
}{
\left(
\rho_{ij}^{(0)}\sigma_{ii}^{(0)} \sigma_{jj}^{(0)} \widehat{{\cal W}}_{\mu,\kappa,\beta_{ij}^{(0)}}(\|\blambda\|)
\right)^2
}
d \blambda
< \infty
\]
which is proved in the proof of Theorem 4 in \cite{bevilacqua2019estimation}.
 \end{proof}
 
  \begin{proof}[{\bf Proof of Theorem \ref{theorem:matern:vs:wendland}}]
From Lemmas \ref{lemma:matern:assympt:invertible}, \ref{lemma:wendland:assympt:invertible} and \ref{lemma:exist:phi_zero:matern}, Condition \ref{cond:phi:zero} holds.
Let $\gamma^2$ be as in Lemma~\ref{lemma:exist:phi_zero:matern}, satisfying Condition~\ref{cond:phi:zero}.
From Theorem~\ref{theorem:sufficient:integral:normalized:difference:small}, in order to prove Theorem~\ref{theorem:matern:vs:wendland}, it is sufficient to show that, for $i,j=1,2$,
\[
\int_{\IR^d}
\frac{
1
}{
\gamma(\blambda)^4
}
\bigl(
\rho_{ij}^{(0)}\sigma_{ii}^{(0)} \sigma_{jj}^{(0)} \widehat{{\cal M}}_{\nu,\alpha_{ij}}(\|\blambda\|)
-
\rho_{ij}^{(1)}\sigma_{ii}^{(1)} \sigma_{jj}^{(1)} \widehat{{\cal W}}_{\mu,\kappa,\beta_{ij}}(\|\blambda\|)
\bigr)^2
d \blambda
< \infty.
\]
This is equivalent to
\[
\int_{\IR^d}
\frac{
\bigl(
\rho_{ij}^{(0)}\sigma_{ii}^{(0)} \sigma_{jj}^{(0)} \widehat{{\cal M}}_{\nu,\alpha_{ij}}(\|\blambda\|)
-
\rho_{ij}^{(1)}\sigma_{ii}^{(1)} \sigma_{jj}^{(1)} \widehat{{\cal W}}_{\mu,\kappa,\beta_{ij}}(\|\blambda\|)
\bigr)
}{
\bigl(
\rho_{ij}^{(0)}\sigma_{ii}^{(0)} \sigma_{jj}^{(0)} \widehat{{\cal M}}_{\nu,\alpha_{ij}}(\|\blambda\|)
\bigr)^2
}
d \blambda
< \infty
\]
which is proved in the proofs of Theorems 5 and 6 in \cite{bevilacqua2019estimation}. In order to apply these proofs, remark that one can show, using the basic properties of the Gamma function, that
$C_{\kappa,\mu}$ is equal to $C_{\nu,\kappa,\mu}$ in Theorem 5 of \cite{bevilacqua2019estimation} if $\kappa>0$ and to $R_{\mu}$ in Theorem 6 of \cite{bevilacqua2019estimation} if $\kappa=0$.
 \end{proof}

\end{document}